\documentclass[11pt]{article}
\usepackage{epsfig,amsmath,amssymb,latexsym,color}
\usepackage{hyperref}
\usepackage{multicol}
\usepackage{multirow}
\usepackage{url}
\usepackage{algorithm}
\usepackage{algorithmic}
\usepackage{booktabs} % for professional tables
\usepackage{wrapfig}
\newcounter{mnote}
  \let\oldmarginpar\marginpar
 \renewcommand\marginpar[1]{\-\oldmarginpar[\raggedleft\footnotesize #1]%
    {\raggedright\footnotesize #1}}
\def\bfb{{\bf b}}

\def\bfe{{\bf e}}
\def\bff{{\bf f}}

\def\bfw{{\bf w}}

\def\bfx{{\bf x}}
\def\bfy{{\bf y}}

\newtheorem{theorem}{Theorem}
\newtheorem{lemma}{Lemma}

\newtheorem{corollary}{Corollary}
\newtheorem{remark}{Remark}
\newtheorem{definition}{Definition}
\newenvironment{proof}{\begin{trivlist}\item[]{\emph{Proof.}}}
               {\hfill$\Box$\end{trivlist}}

\begin{document}
\title{The Kolmogorov Superposition Theorem can Break the Curse of Dimensionality
when Approximating High Dimensional Functions} 
\author{Ming-Jun Lai\footnote{mjlai@uga.edu. Department of Mathematics,
University of Georgia, Athens, GA 30602. This author is supported by the Simons Foundation 
Collaboration Grant \#864439.} 
\and 
Zhaiming Shen\footnote{zshen49@gatech.edu.
Department of Mathematics,
University of Georgia, Athens, GA 30602.}}
\maketitle

% \leftline{Mathematical Classification(2020): 41A15. 41A25. 41A63,65D07,65D40}
\begin{abstract} 
% Recall when approximating a high dimensional 
% continuous function $f\in C[0,1]^d$ by polynomials or tensor product B-splines, we need $O(n^d)$ data values to have an error $O(1/n)$.  This phenomenon is the so-called curse of dimensionality. 
We explain how to use Kolmogorov Superposition Theorem (KST) to break the curse of dimensionality when approximating a dense class of multivariate continuous functions. We first show that there is a class of functions called Kolmogorov-Lipschitz (KL) continuous in $C([0,1]^d)$ which can be approximated by a special ReLU neural network of two hidden layers with a dimension independent approximation 
rate $O(1/n)$ with approximation constant increasing quadratically in $d$. The number of parameters used in such neural network approximation equals to $(6d+2)n$. 
Next we introduce KB-splines by using linear B-splines to replace the outer function  and smooth the KB-splines to have the so-called LKB-splines as the basis for approximation. 
Our numerical evidence shows that  
the curse of dimensionality is broken in the following sense: 
When using the standard discrete least squares (DLS) method to approximate a continuous function, 
there exists a pivotal set of points in $[0,1]^d$ 
with size at most $O(nd)$ such that the rooted mean squares error (RMSE) from the DLS based on the pivotal set is similar to the RMSE of the DLS based on the 
original set with size $O(n^d)$. The pivotal point set is chosen by using matrix cross approximation technique 
and the number of LKB-splines used for approximation is the same as the size of the pivotal data set. 
Therefore, we do not need too many basis functions nor too many function values 
to approximate a high dimensional continuous function $f$.   Hence, the study in this paper provides an approach for dimension reduction problems. 
%Finally, we present a mathematical justification for image classification 
%by using a deep learning algorithm.   
\end{abstract} 


\section{Introduction}
Recently, deep learning algorithms have shown a great success in many fronts of research, 
from image analysis, audio analysis, biological data analysis, to name a few.  
Incredibly, after a deep learning training of thousands of images,  
a computer can tell if a given image   is a cat, or a dog, or neither of them 
with very reasonable accuracy. 
In addition, there are plenty of successful stories such that deep learning algorithms 
can sharpen,  denoise,  enhance an image  after an intensive training. See, e.g. \cite{G16}
and \cite{M16}.  The 3D structure of a DNA can be predicted very accurately by using the 
DL approach. 
The main ingredient in DL algorithms is the neural network approximation based on ReLU functions. 
We refer to \cite{DD19} and \cite{DHP21} for
detailed explanation of the neural network approximation in deep learning algorithms and the literature therein. 

Learning a multi-dimensional data set is like approximating a multivariate function. 
The computation of a good approximation suffers from the curse of dimension.  
For example, suppose that  $f\in C([0, 1]^d)$ with $d\gg 1$.
One usually uses Weierstrass theorem to have a polynomial $P_f$ of degree $n$ such that 
$$
\|f- P_f\|_\infty \le \epsilon
$$ 
for any given tolerance $\epsilon>0$. As the dimension of polynomial 
space $= {n+d\choose n}\approx n^d$ when
$n >d$, one will need at least $N=O(n^d)$ data points in $[0, 1]^d$ to distinguish 
different polynomials in $\mathbb{P}_n$ 
and hence, to determine this $P_f$. Notice that many good approximation schemes enable $P_f$ to 
approximate $f$ at the rate of $O((1/n)^{m})$ if $f$ is of $m$ times differentiable.   
In terms of the number $N$ of data points which should be greater than or equal to the 
dimension of polynomial space $\mathbb{P}_n$,  i.e. $N\ge n^d$, 
the order of approximation is $O(1/(N^{1/d})^m)$. When $d\gg 1$ is bigger, 
the order of approximation is {less}. This phenomenon is the well-known curse of dimensionality. 
Sometimes, such a computation is also called intractable. 

Similarly, if one uses a tensor product of B-spline functions to approximate 
$f\in  C([0, 1]^d)$,  one needs to subdivide the $[0, 1]^d$ into $n^d$ small 
subcubes by hyperplanes parallel to the axis planes.  As over each small subcube, 
the spline approximation $S_f$  of $f$ is a polynomial of
degree $k$, e.g., $k=3$ if the tensor product of cubic splines are used. 
Even with the smoothness, one needs $N=O(k^d)$ data points and function values at these data 
points in order to determine a polynomial piece of $S_f$ over each subcube. 
Hence, over all subcubes, one needs $O(n^dk^d)$.   
It is known that the order of approximation of $S_f$ is $O(1/n^{k+1})$ if $f$ 
is of $k+1$ times coordinatewise  differentiable.  In terms of $N=O(n^dk^d)$ points 
over $[0, 1]^d$, the approximation order of $S_f$ will be  $O(k^{k+1}/N^{(k+1)/d})$.  
More precisely,  in \cite{DHM89}, the researchers showed that 
the approximation order $O(1/N^{1/d})$ can not be improved 
for smooth functions in Sobolev space $W^{k,p}$ with $L_p$ norm  $\le 1$.
In other words, the approximation problem by using multivariate polynomials or by tensor product 
B-splines is intractable.    

Furthermore, many researchers have worked on using ridge functions, deep neural networks, and 
ReLU  and many newly invented activation functions to approximate functions in high dimension.
See \cite{M99,P99,P98,B17} for detailed statements and proofs.  
More recently, the super approximation power was introduced in \cite{SYZ21} which uses the floor function, exponential function, step function, and their compositions as the activation function and can achieve the exponential approximation rate.
Although the number of the depth and the number of the width of deep neural networks are  
quadratically dependent on the dimensionality $d\ge 2$ (cf. \cite{SYZ22}), any training of such a neural network requires an 
exponentially many data locations and function values.  
That is, the approximation problem by using these neural networks is intractable.

Let us summarize the bottleneck problem of approximating high dimensional functions as follows. 
The curse of dimensionality (COD) contains three difficult components: (1) one has to use exponentially many basis functions when approximating a high dimensional function; (2) one has to use exponentially many  values from the function to be approximated; 
and (3) the rate of convergence decreases to zero as $d\to \infty$.   In particular, it is hard or very 
expensive to obtain  an exponential amount of high dimensional data (points and functions). 
Such an amount of data is even too much for the current computer storage.  

 However, there is a way to overcome the first and third componets of the curse of dimensionality  as explained by Barron in \cite{B93}. 
Let $\Gamma_{B,C}$ be the class of functions $f$ defined over $B=\{\bfx\in \mathbb{R}^d, \|\bfx\|\le 1\}$ 
such that $C_f\le C$, where $C_f$ is defined as follows:  
$$
C_f= \int_{\mathbb{R}^d} |\omega| |\widetilde{f}(\omega)| d\omega
$$
with $|\omega|= (\omega\cdot \omega)^{1/2}$ and $\widetilde{f}$ is the 
Fourier transform of $f$. 
\begin{theorem}  [Barron, 1993]
\label{Barron93}
For every function  $f$ in $\Gamma_{B,C}$, every sigmoidal 
function $\phi$,  every  probability  measure  $\mu$,  and  every  $n  \ge   1$, 
there exists a linear combination  of  sigmoidal functions $ f_n(x)$, a shallow neural network,  
such  that  
\begin{equation}
    \label{BarronRate}
\int_B (f(\bfx) - f_n(\bfx))^2 d\mu(\bfx) \le \frac{(2C)^2}{n}.
\end{equation}
The  coefficients  of  the  linear  combination  in  $f_n$  may  be 
restricted  to  satisfy  $\sum_{k=1}^n |c_k|\le  2C$,  and  $c_0 = f(0)$. 
\end{theorem}

The Barron result shows that there exists an approximation scheme which only needs $n$ terms of 
sigmoidal functions $f_n$ and achieves the order $O(1/n)$ independent of $d$. However, it does not known 
how to obtain this approximation scheme without using an exponentially many data values.   
It is worth noting that although the approximation rate is independent 
of the dimension in the $L_2$ norm sense, the constant $C$ can be exponentially large in the worst case scenario 
as pointed out in \cite{B93}.   This work leads to many 
recent studies on the properties and structures of Barron space $\Gamma_{B,C}$ 
and its extensions, e.g. the spectral Barron spaces and the generalization using ReLU 
or other more advanced activation functions instead of sigmodal functions
above, e.g., \cite{KB18}, \cite{E20}, \cite{EMW19}, \cite{EW20}, \cite{EW20b}, \cite{SX20}, 
\cite{SX20b}, and the references therein. To the best of our knowledge, how to design algorithms to achieve the rate of convergence in 
(\ref{BarronRate}) numerically is still under study. 

For practical purposes,  to overcome the curse of dimensionality, let us use the following definition. 
\begin{definition}
\label{COD}
Fix a discrete semi-norm of interest, e.g. $\|f\|_{\cal P} =  \sqrt{ \sum_{i=1}^N (f(\bfx_i))^2/N}$, where 
$\bfx_i, i=1, \cdots, N$ are points in $[0, 1]^d$.  
When approximating a continuous function $f\in C([0,1]^d)$, if there is a scheme $S_n$ based on
$O(nd)$ basis functions and function values $f(\bfx_i), i=1, \cdots, O(nd)$ with 
$\bfx_i\in [0, 1]^d$ such that 
the error $\|f- S_n\|_{\cal PP}\le C(1/n^\alpha)$ for a positive number $\alpha$ independent of $d$, and 
a positive constant $C$ which may be dependent on $d$ quadratically 
and dependent on $f$, but independent of $n$, then we say the scheme $S_n$ does not suffer from the curse of
dimensionality, 
where $\|\cdot \|_{\cal PP}$ is another discrete semi-norm based on  much dense points than the points $\bfx_i$'s in $[0, 1]^d$. 
\end{definition}
Let us comment that the number of data (points and function values) has to be dependent on the dimensionality. 
For example, to approximate a linear polynomial $p$ in $d$-dimensional space, we need $(d+1)$ points and its
values in order to have a good approximation in general. Also, function $f$ can grows very quickly as 
and/or oscillate very quickly $d\to \infty$. There are no way to use such a small number of data (points and 
function values), i.e. $O(nd)$ samples of the values of a function $f(\bfx)=\|\bfx\|^{d^d}$ to approximate it. Thus, we have to restrict ourselves to some classes of functions  which are bounded and/or oscillate 
grow moderately. These will be more precisely described later in the paper.

In this paper, we turn our attention to Kolmogorov superposition theorem (KST) and will see how 
it plays a role in the study how to approximate high dimensional continuous function and how it can 
be used to break the  curse of dimensionality when approximating high 
dimensional functions which grows and/or oscillate moderately 
for $d=2,3$ (this paper) and $d=4, 5, 6$ (cf. \cite{S24}),  
%and how it can potentially overcome the curse of dimension when $d\ge 7$.  
%Let us be precise on the curse of dimensionality. 

For convenience, we
will use KST to stands for Kolmogorov superposition theorem for the rest of the paper. 
Let us start with the statement of KST using the version of G. G. Lorentz from  \cite{L62} 
and \cite{L66}.

\begin{theorem}[Kolmogorov Superposition Theorem] 
\label{kolmogorov}
There exist irrational numbers $0<\lambda_p\leq 1$ for $p=1,\cdots, d$ and strictly increasing 
$Lip(\alpha)$ functions $\phi_q(x)$ 
(independent of $f$) with $\alpha=\log_{10}2$  
defined on $I=[0,1]$ for $q=0, \cdots, 2d$ such that 
for every continuous function $f$ defined on $[0,1]^d$, 
there exists a continuous function $g(u)$, $u\in [0,d]$ such that 
\begin{equation}
\label{repres}
    f(x_1,\cdots,x_d)=\sum_{q=0}^{2d} g(\sum_{i=1}^d\lambda_i\phi_q(x_i)).
\end{equation}
\end{theorem}

\noindent
Note that $\phi_q, q=0, \cdots, 2d$ are called inner functions while $g$ 
is called outer function. In the original version of KST (cf. \cite{K56}, \cite{K57}), 
there are $2d+1$ outer functions and $d(2d+1)$ inner functions.  
G. G. Lorentz simplified the construction to have one outer function and $2d+1$ inner functions 
with $d$ irrational  numbers $\lambda_i$. See \cite{MP99} for KST with one outer function and $d(2d+1)$
inner functions.  
It is known that $2d+1$ can not be smaller (cf. 
\cite{D63}) and 
the smoothness of the inner functions in the original version of the KST can be improved to be 
Lipschitz continuous (cf. \cite{F67}) while the continuity of $\phi_q$ in the Lorentz version 
can be improved to be $Lip_\alpha$ for
any $\alpha\in (0, 1)$ as pointed out in \cite{L66} at the end of the chapter.   
There were Sprecher's constructive proofs (cf. \cite{S63, S65a, S65b, S66, S72, S93, S96, S97})
which were finally correctly established in \cite{B09} and \cite{BG09}.  An excellent 
explanation of the history on the KST can be found in \cite{M21} together with a constructive 
proof similar to the Lorentz construction.

Recently, KST has been extended to unbounded domain in \cite{L21} and has been actively studied   
during the development of the neural network computing (cf. e.g. \cite{C89}, \cite{MM92}, 
\cite{P99}, \cite{MY20}) as well as during the fast growth period of  
the deep learning computation recently.  Hecht-Nielsen \cite{H87} is among the first to explain 
how to use the 
KST as a feed forward neural network in 1987. In particular, Igelnik 
and Parikh \cite{IP03} in 2003 proposed an algorithm of the neural network using spline functions
to approximate both the inner functions   and outer function $g$.  
More recently, in \cite{KAN24}, the researchers have created a new way to implement two layer neural networks inspired by the Kolmogorov superposition theorem, which experimentally has a better performance than the two Multi-Layer Perceptrons (MLPs).
See also \cite{FFS22} for another efficient
implementation and numerical results. 

It is easy to see that the formulation of the KST 
is similar to the structure of a neural 
network where the inner and outer functions can be thought as two hidden layers when 
approximating a continuous function $f$.  
However, one of the main obstacles is that the outer function $g$ depends on  $f$ 
and furthermore, $g$ can vary wildly even if $f$ is smooth as explained in \cite{Girosi1989}.  
Schmidt-Hieber \cite{Schmidt2021} used a version of KST in \cite{B09} and approximated 
the inner functions by using a neural network of multiple layers to form a deep 
learning approach for approximating any H\"older continuous functions.  
Montanelli and Yang in \cite{MY20} used very deep neural networks to approximate the 
inner and outer functions to obtain the approximation rate   $O(n^{-\frac{1}{\log d}})$ 
for a function $f$ in the class $K_C([0,1]^d)$ (see \cite{MY20} for detail).  
That is, the curse of dimensionality  is lessened.    To the best of our knowledge, 
the approximation bounds in the existing work for general function still suffer the 
curse of dimensionality unless $f$ has a simple outer function $g$ as explained in this paper, 
and how to characterize the class of functions which has moderate outer functions remains an 
open problem. However, in the last section of this paper, we will provide a numerical method to test if a 
multi-dimensional continuous function is Kolmogorov-Lipschitz continuous.

The development of the neural networks obtains a great speed-up 
by using the ReLU function instead of a signmoidal activation function. 
Sonoda and Murata1 \cite{SM15} showed that the ReLU is an approximator and 
established the universal approximation theorem, i.e. Theorem~\ref{SM15} below.  
Chen in \cite{C16} pointed out that the computation of  layers in the deep learning algorithm 
based on ReLU is a linear spline and studied the upper bound on the number of knots needed 
for computation.  Motivated by the work from \cite{C16}, 
Hansson and Olsson in \cite{HO17} continued the study and gave another justification for the 
two layers of deep learning are  a combination of linear spline   and used 
Tensorflow to train a deep learning machine for approximating the well-known Runge function 
using a few knots.  There are several dissertations written 
on the study of the KST for the neural networks and
data/image classification. See \cite{B08},  \cite{B09}, \cite{F10}, \cite{L15}, and etc..  

The main contribution of this work is  we propose to study the rate of approximation for ReLU neural network through KST. We propose a special neural network structure via the representation of KST, which can achieve a dimension independent approximation rate 
$O(1/n)$ with the approximation constant increasing quadratically in the dimension when approximating a dense subset of continuous functions. The number of parameters used in such a network increases linearly in $n$. Furthermore, we shall provide a numerical scheme to practically approximate $d$ dimensional continuous functions by using at most $O(dn)$ number of pivotal locations for function value 
evaluattion instead of the whole equally-spaced $O(n^d)$ data locations, and such a set of pivotal locations are independent of 
target functions. 

The subsequent sections of this paper are structured as follows. 
In section \ref{sec2}, we explain how to use KST to approximate multivariate continuous functions without the curse of dimensionality for a certain class of functions. We also establish the approximation result for any continuous function based on the modulus of continuity of the outer function. 
In section \ref{sec3}, we introduce KB-splines and its smoothed version LKB-splines. We will show that 
KB-splines are indeed the basis for functions in $C([0,1]^d)$. 
In section \ref{sec4}, we numerically demonstrate in 2D and 3D that LKB-splines can approximate functions in $C([0,1]^d)$ very well. Furthermore, we provide a computational strategy based on matrix cross approximation to find a sparse solution using a few number of LKB-splines to achieve the same approximation order as the original approximation.  This leads to 
the new concept of pivotal point set from any dense point set $P$ over $[0,1]^d$ such that the discrete least squares (DLS) fitting based on the pivotal point set has the similar rooted mean
squares error to the DLS fitting based on the original data set $P$. 
%As the size of the magic set is much smaller than the size of $P$, one can approximate any continuous function using its function values over the magic point set and hence, one is able to break the curse of dimension. Numerical evidence will be provided for 2D and 3D settings to show that the DLS based on the magic point sets can well approximate continuous, not very oscillating functions.  

\section{Universal Approximation via KST}
\label{sec2}
{We will use $\sigma_1$ to denote ReLU function through the rest of discussion.}
It is easy to see that one can use linear splines to approximate 
inner (continuous and monotone increasing) functions $\phi_q, 
q=0, \cdots, 2d$ and also approximate the outer (continuous) function $g$. We refer to 
Theorem 20.2 in \cite{P81}. 
 On the other hand, 
we can easily see that any linear spline function can be written in terms of linear 
combination of ReLU functions and vice versa, see, e.g. \cite{DD19}, and \cite{DHP21}. 
(We shall include another proof later in this paper.) Hence, we have 
$$
L_q(t):=\sum_{j=1}^{N_q}c_{q,j}\sigma_1(t-y_{qj}) \approx \phi_q(t)
$$
for $q=0, \cdots, 2d$ and
$$
S_g(t):= \sum_{k=1}^{N_g} w_k \sigma_1(t-y_k) \approx g,
$$
where $g$ is the outer function of a continuous function $f$. Based on KST and the universal approximation theorem (cf. \cite{C89}, \cite{Hornik1991}, \cite{P99}), it follows that
\begin{theorem}[Universal Approximation Theorem (cf. \cite{SM15}]
\label{SM15}
Suppose that $f\in C([0, 1]^d)$ is a continuous function. For any given $\epsilon>0$, there 
exist coefficients $w_k, k=1, 
\cdots, N_g$, $y_k \in [0, d], k=1, \cdots, N_g$, $c_{q,j}, j=1, \cdots, N_i$ 
and $y_{q,j} \in [0, 1], j=1, \cdots, N_i$  such that 
\begin{equation}
\label{splineapp2}
|f(x_1, \cdots, x_d)  -   \sum_{q=0}^{2d} \sum_{k=1}^{N_g} w_k \sigma_1(\sum_{i=1}^d \lambda_i 
\sum_{j=1}^{N_q}c_{q,j}\sigma_1(x_i-y_{qj})-y_k)|\le \epsilon. 
\end{equation}
\end{theorem}   
\noindent
{In fact, many results similar to the above (\ref{splineapp2}) have been established using other activation functions
including learnable activation functions (cf. e.g. \cite{C89}, \cite{K92}, \cite{MY20}, \cite{KAN24}, and etc.).} \textcolor{black}{We shall present a rate of convergence in this section. See 
(\ref{subcurse}). }

% Similarly, if we use high degree activation function $\sigma_\ell(t)= (t_+)^\ell, \ell\ge 2$, 
% we extend the ideas above to immediately  have a more general version:
% \begin{theorem}[Universal Approximation Theorem using $\sigma_\ell$]
% Suppose that $f\in C([0, 1]^d)$ is a continuous function. For any given $\epsilon>0$, there 
% exist coefficients $w_k, k=1, 
% \cdots, N_g$, $y_k \in [0, d], k=1, \cdots, N_g$, $c_{q,j}, j=1, \cdots, N_q$ 
% and $y_{q,j} \in [0, 1], j=1, \cdots, N_q$  such that 
% \begin{equation}
% \label{splineapp3}
% |f(x_1, \cdots, x_d)  -   \sum_{q=0}^{2d} \sum_{k=1}^{N_g} w_k \sigma_\ell(\sum_{i=1}^d \lambda_i 
% \sum_{j=1}^{N_q}c_{q,j}\sigma_\ell(x_i-y_{qj})-y_k)|\le \epsilon. 
% \end{equation}
% \end{theorem}  

\subsection{Kolmogorov-Lipschitz (KL) Functions}
To establish the rate of convergence for Theorem~\ref{SM15}, we introduce a new concept called Kolmogorov-Lipschitz (KL)
continuity. 
For each continuous function $f\in C([0, 1]^d)$, 
let $g_f$ be the outer function associated with $f$. We define 
\begin{equation}
\label{G1}
{\rm KL}:=\{f:  \hbox{ the outer function $g_f$ is Lipschitz continuous over $[0, d]$} \}
\end{equation}
be the class of Kolmogorov-Lipschitz (KL) continuous functions.  
Note that when $f$ is a constant, its outer function $g=\frac{1}{d+1}f$ 
is also constant (cf. \cite{B08}) and hence, is
Lipschitz continuous.  That is, the function class KL is not empty.  Furthermore, 
we can use any univariate Lipschitz continuous function $g$ such as $g(t)=Ct, g(t)=\sin(Ct), g(t)=\exp(-Ct)$,  
$g(t)=\sin(Ct^2/2)$, etc..
over $[0, d]$ to define a multivariate  function $f$ by using the formula (\ref{repres}) of KST, where $C$ is any constant. 
Then these 
newly defined $f$ are continuous over $[0, 1]^d$, and belong to the function class KL with Lipschitz constant $C$
which is independent of the dimensionality $d$.  

On the other hand, we can use $g(t)=t^d$ or $g(t)=t^{d^d}$ or $g(t)=\sin(d^d t\pi)$ over $[0, d]$ to define a multivariate continuous function by the KST formula 
(\ref{repres}) in $C[0, 1]^d)$. These functions either grow exponentially when $d\ge 3$ is large or oscillate very quickly  that spoil the fun of constructing an approximation scheme which overcomes the curse of dimensionality. In general, such pathological functions will seldom appear in practice.  
For practical purposes, we shall restrict some classes of multidimensional functions to avoid 
these pathological functions. To do so, we let $\bigcup_{d=2}^\infty C([0, 1]^d)$ be the algebraic union of these continuous function spaces
$C([0, 1]^d)$'s.  Note that any function in $C([0,1]^d)$ can be extended naturally to be a function in $C([0,1]^{d+1})$ and hence, can be extended to be
a continuous function in infinitely many variables.  Hence, we define by 
\begin{equation}
    \label{newfunctionspace}
    C([0,1]^\infty) = \hbox{closure of } \bigcup_{d=2}^\infty C([0, 1]^d) \hbox{ in the $L^\infty$ norm}
\end{equation}
the Banach space of the maximal norm completion of $\bigcup_{d=2}^\infty C([0, 1]^d)$.  
Furthermore, we say a function $f\in C([0, 1]^\infty)$ is 
a Kolmogorov-Lipschitz (KL) continuous function with KL constant $C$ if there is a sequence of KL continuous functions $f_n$ 
in $C([0, 1]^{d_n})$ with KL constant $\le C$ such that $f_n\to f$ in the maximal norm.  Let 
$KL([0,1]^\infty)$ be the class of all Kolmogorov-Lipschitz functions in $C([0, 1]^\infty)$.  
We shall show that these functions in $KL([0,1]^\infty)$ will be dense in  $C([0, 1]^\infty)$. \textcolor{black}{
We shall use this space $ C([0,1]^\infty)$ to explain the curse of dimensionality later in the paper.}

It is easy to see that the class KL is dense in $C([0,1]^d)$ by Weierstrass approximation theorem for any 
fixed dimensionality $d>1$. See Theorem~\ref{LS23} below.  
For another example, let $g(t)$ be a B-spline function of degree $k\ge 1$, the associated multivariate function is in KL. We shall use such B-spline functions for the outer function $g$ approximation in a later section. 

\begin{theorem}
\label{main0}
\label{LS23} Fix any integer $d\ge 2$. 
For any $f\in C([0,1]^d)$ and any $\epsilon>0$, there exists a KL continuous function $K_\epsilon$ dependent on $\epsilon$ such that 
\begin{equation}
    \|f-K_\epsilon\|_{\infty}\leq \epsilon.
\end{equation}
Furthermore, the Kolmogorov-Lipschitz (KL) constant $C$ of $K_\epsilon$ is  $C = O(n^2 4^{2n}\|g_f\|_{\infty}/d)$, where $g_f$ is the outer
function of $f$ and $n$ is dependent on $\epsilon$ as described in the proof. 
\end{theorem}
\begin{proof}
    By the KST, we can write $f(x_1,\cdots,x_d)=\sum_{q=0}^{2d+1}g(\sum_{i=1}^d\lambda_i\phi_q(x_i))$. By Weierstrass approximation theorem, there exists a polynomial $p_\epsilon$ such that $|p_\epsilon(t)-g(t)|\leq\frac{\epsilon}{2d+1}$ 
    for all $t\in [0,d]$. Such a polynomial $p_\epsilon$ is certainly a  Lipschitz continuous function over $[0, d]$ whose Lipschitz constant will be given soon.   
    
    Therefore, by letting $K_\epsilon(x_1,\cdots,x_d)=\sum_{q=0}^{2d+1}p_\epsilon(\sum_{i=1}^d\lambda_i\phi_q(x_i))\in KL$, we have a simple formula
\begin{align*}
    |f(x_1,\cdots,x_d)-K_{\epsilon}(x_1,\cdots,x_d)|&=|\sum_{q=0}^{2d+1}g(\sum_{i=1}^d\lambda_i\phi_q(x_i))-\sum_{q=0}^{2d+1}p_\epsilon(\sum_{i=1}^d\lambda_i\phi_q(x_i))|  \\
    &\leq \sum_{q=0}^{2d+1}|g(\sum_{i=1}^d\lambda_i\phi_q(x_i))-p_\epsilon(\sum_{i=1}^d\lambda_i\phi_q(x_i))| \\
    &\leq (2d+1)\cdot \frac{\epsilon}{(2d+1)} =\epsilon.
\end{align*}
To specify the Lipschitz constant of the polynomial $p_\epsilon$, we recall the well-known Fej\'er-Hermite 
polynomials which is defined by
\begin{equation}
\label{HFoperator}
F_n(g;x) = \sum_{i=1}^n g(x_i) \bigl[1-2Dl_i(x_i) (x-x_i)\bigr]L_i^2(x)
\end{equation}
where $L_i(x) = \displaystyle \prod_{j=1\atop j\not=i}^n \frac{(x-x_j)}{
(x_i-x_j)}, i=1,\cdots, n $ and $\{x_1,\cdots,x_n\}\subseteq [a,b]$ is a set
of distinct knots. We refer to \cite{C66} for detail.  
Furthermore, for $[a,b]=[-1, 1]$ and the distinct knots  chosen to be the zero of Tchebysheff polynomial $T_n$ 
of degree $n$, we have
\begin{eqnarray*}
F_n(g;x) &=& {T_n^2(x)\over n^2} \sum_{i=1}^n g(x_i) {1-xx_i\over (x-x_i)^2},
\end{eqnarray*}
where $x_i = \cos (2i-1) \pi/2n,  i=1,\cdots,n.$  (See, p. 79 of \cite{C66}.)  When $g\in C[-1,1]$, 
since $F_n(1;x) =1,$ for $x\in [-1,1]$, we have
$$
F_n(g;x)-g(x) = {T_n^2(x)\over n^2} \sum_{i=1}^n (g(x_i)-g(x)) {1-xx_i\over (x-x_i)^2}.
$$
\begin{eqnarray*}
\bigl|F_n(g;x)-g(x)\bigr| &\le &{T^2_n(x)\over n^2}\Biggl(\sum_{|x-x_i|<\delta}
\epsilon {1-xx_i\over (x-x_i)^2} + \sum_{|x-x_i|\ge\delta}
2\Vert g\Vert_\infty {1-xx_i\over (x-x_i)^2} \Biggr) \\
&\le & {T_n^2(x)\over n^2} \Biggl(\epsilon \sum_{i=1}^n {1-xx_i\over (x-x_i)}
+ {2\Vert g\Vert_\infty\over \delta^2}\sum_{i=1}^n (x-x_i)^2 {1-xx_i\over
(x-x_i)^2}\Biggr)\\
&=&\epsilon + {2\Vert g\Vert_\infty\over \delta^2 n^2} \sum_{i=1}^n (1-xx_i) 
\le \epsilon + {4\Vert g\Vert_\infty \over \delta^2n^2}\cdot n \le 
 2\epsilon \qquad \hbox{ when $n$ large enough} . 
\end{eqnarray*}
In fact, $n$ is so large such that $n\ge \sqrt{ 2\|g\|_\infty/(\delta^2\epsilon)}$. 

Let us convert the polynomial $F_n$ to be over $[0, d]$ with $g=g_f\in C[0, d]$, the above estimate can be carried
out as well to establish the convergence.  
Now let us  estimate its Lipschitz constant of $F_n$. It is easy to see 
the Fej\'er-Hermite polynomial is $F_n(g; -1+ 2x/d)$ over $[0, d]$ and more precisely,  
\begin{eqnarray*}
F_n(g;-1+2x/d) &=& {T_n^2(-1+2x/d)\over n^2} \sum_{i=1}^n g_f(y_i) {2xd +2dy_i-4xy_i \over 4(x-y_i)^2},
\end{eqnarray*}
where $y_i=d(x_i+1)/2\in [0, d]$ for $i=1, \cdots, n.$ If we write $F_n(g; -1+2x/d)= a_0x^{2n-1}+ a_1x^{2n-2}+\cdots$, then 
$$
a_0 = \lim_{x\to \infty} \frac{ F_n(g; -1+2x/d)}{x^{2n-1}} = \frac{1}{n^2} 
\lim_{x\to \infty}  {2^{2n}T_n^2(-1+2x/d)\over d^{2n}(2x/d-1)^{2n} } 
\sum_{i=1}^n g(y_i) {x(2xd +2dy_i-4xy_i) \over 4(x-y_i)^2}
$$
and hence, $|a_0| \le 4^{2n}/d^{2n} \Vert g\Vert_\infty \sum_{i=1}^n |d/2-y_i|$ since the leading coefficient of
$T_n$ is $2^n$. We know that the Lipschitz constant 
$C$ of $F_n(g;-1+2x/d)$ is bounded by $\|F'_n(g;\cdot)\|_\infty\le (2n-1)(|a_0|d^{2n-2}+|a_1|d^{2n-3}+\cdots ) =
O((2n-1) |a_0| d^{2n-2})$ when $d\gg 2$. It now follows that the Lipschitz constant $C$ of $F_n$ is 
\begin{equation}
\label{keyconstant}
    C \le \frac{(2n-1) d^{2n-2} 4^{2n}}{d^{2n}}\|g\|_{\infty} nd \le (2n^2)4^{2n} \|g\|_{\infty}/d.
\end{equation}
These finish the proof. 
\end{proof}

We are now ready to establish one of our main results in this paper. 
\begin{theorem}
\label{main2}
The class ${\rm KL}([0,1]^\infty)$ is dense in $C([0,1]^\infty)$. 
\end{theorem}
\begin{proof}
    Suppose that $f\in C([0,1]^\infty)$.  For any $\epsilon>0$, we first find $f_d
    \in C([0,1]^d)$ such that $\|f - f_d\|_\infty \le \epsilon/2$.  Since $f$ is bounded, so
    is $f_d$.  We use Theorem~\ref{main0} to have a KL polynomial $P_\epsilon$ such that
    $$ \|f_d - P_\epsilon\|_\infty \le \epsilon/2.$$
    It follows that $\|f- P_\epsilon\|_\infty \le \|f- f_d\|_\infty + \|f_d - P_\epsilon\|_\infty
    \le \epsilon$.  
    Note that based on 
    the constructive proof of Kolmogorov superposition theorem, we know the outer function 
    $g_f$ is bounded and hence, the Kolmogorov-Lipschitz constant $C$ is bounded \textcolor{black}{independent of
    $d$} based on the estimate in (\ref{keyconstant}).  
\end{proof}

%If $f$ is a continuous function with K-outer function $g_f$ bounded independent of $d$, we know $K_\epsilon$ 
%is a KL function with Lipschitz constant $C$ above.  
%If $f$ is constinuous function with $g_f$  growing like $O(d^\alpha)$ for a fixed constant $\alpha>0$,  

Based on the discussion above, we now focus on the multidimensional functions which are Kolmogorov-Lipschitz continuous functions.  

\section{Approximation via ReLU Neural Networks}
In this section, we study the approximation of multivariate continuous functions by using the standard ReLU neural networks. 
Note that the neural network being used for approximation in expression (\ref{splineapp2}) is a special class of neural network with two hidden layers of 
widths $(2d+1)dN_q$ and $(2d+1)N_g$ respectively. Let us call this special 
class of neural networks the 
Kolmogorov network, or K-network in short and use $\mathcal{K}_{m,n}$ to denote the K-network 
of two hidden layers with widths $(2d+1)dm$ and $(2d+1)n$ based on ReLU activation function, i.e.,
\begin{equation}
\mathcal{K}_{m,n}(\sigma_1) =\{
\sum_{q=0}^{2d}\sum_{k=1}^{dn} w_k \sigma_1(\sum_{i=1}^d 
\sum_{j=1}^{m}c_{qj}\sigma_1(x_i-y_{qj})-y_k), w_k, c_{qj}\in \mathbb{R}, y_k\in [0, d], 
y_{qj}\in [0, 1]\}.
\end{equation}
The parameters in $\mathcal{K}_{m,n}$ are $w_k,y_k$, $k=1,\cdots,dn$, and $c_{qj},y_{qj}$, $q=0,\cdots,2d$, $j=1,\cdots,m$. Therefore the total number of parameters equals to $2dn+2(2d+1)m$.
In particular if $m=n$, the total number of parameters in this network is  $(6d+2)n$. 
% \noindent
% Moreover, we use $\mathcal{N}_{m,n}$ to denote the standard ReLU neural network of two hidden layers 
% with widths $m$ and $n$, i.e.,
% \begin{equation}
% \label{nn}
% \mathcal{N}_{m,n}(\sigma_1) =\{
% \sum_{k=1}^{n} w_k \sigma_1(\sum_{j=1}^m 
% c_{kj}\sigma_1(\mathbf{w}_j^{\top}\mathbf{x}-y_{j})-z_k): w_k, c_{j}, z_k, y_j \in \mathbb{R}, \mathbf{w}_j\in\mathbb{R}^d\}. 
% \end{equation}
% It is easy to see that $\mathcal{K}_{m,n}\subset \mathcal{N}_{(2d+1)dm, n}$. }
We are now ready to state another one of the main results in this paper. 

\begin{theorem} 
\label{mjlai2019}
\textcolor{black}{Let $f\in C([0, 1]^d)$ and assume $f$ is in the class $KL([0,1]^\infty)$.}  
Let $C_f$ be the Kolmogorov-Lipschitz constant of $f$, i.e. the Lipschitz constant of the outer function $g_f$ associated with $f$.
We have 
\begin{equation}
\label{subcurse}
\inf_{ s\in \mathcal{K}_{n, n}(\sigma_1) } 
\|f - s\|_{C([0,1]^d) } \le \frac{C_f(2d+1)^2}{n}.
\end{equation}
% where $\mathbb{N}_{dn, n}(\sigma_1)$ is defined in (\ref{nn}). 
\end{theorem}

%\noindent
%\textcolor{blue}{
The significance of this result is that we need only $(6d+2)n$ parameters to achieve the approximation rate $O(1/n)$ with the approximation constant increasing quadratically in the dimension if $C_f$ bounded or growing moderately in $d$.  
To overcome the curse of dimensionality when approximating $f$, we shall discuss how to use $O(nd)$ data points and values to achieve the order 
of the convergence in (\ref{subcurse}) in the later sections.

%\subsection{Proof of Theorem~\ref{mjlai2019}}
To prove Theorem~\ref{mjlai2019}, we need some preparations.
Let us begin with the space $\mathcal{N}(\sigma_1)=\hbox{span}\{ \sigma_1(\bfw^\top\bfx- b), 
b\in \mathbb{R}, \bfw\in \mathbb{R}^d\}$ which is the space of  shallow networks of 
ReLU functions.  
It is easy to see that all linear polynomials over $\mathbb{R}^d$ are in $\mathcal{N}(\sigma_1)$. The following
result is known (cf. e.g. \cite{DD19}).  For self-containedness, we include a different proof.  
\begin{lemma}
	For any linear polynomial $s$ over $\mathbb{R}^d$, there exist coefficients 
	$c_i \in \mathbb{R}$, bias $t_i\in \mathbb{R}$ and weights  
	$\bfw_i\in \mathbb{R}^d$ such that 
	\begin{equation}
	\label{mjlai2}
	s(\bfx) = \sum_{i=1}^n c_i\sigma_1( \bfw_i \cdot \bfx + t_i),  \forall \bfx \in [0, 1]^d.
	\end{equation}
	That is, $s\in \mathcal{N}(\sigma_1)$.  
\end{lemma}
\begin{proof}
	It is easy to see that a linear polynomial $x$ can be exactly reproduced by using the ReLU functions.  For example,   
	\begin{equation}
	\label{key3}
	x= \sigma_1(x), \forall x\in [0, 1].
	\end{equation} 
	Hence, any component $x_j$ of $\bfx\in \mathbb{R}^d$ can be written in terms of (\ref{mjlai2}).
	Indeed, choosing $\bfw_i=\bfe_j$, from (\ref{key3}), we have 
	$$
	x_j=   \sigma_1(\bfe_j \cdot \bfx), \quad \bfx\in [0, 1]^d, j=1, \cdots, d.
	$$
	Next we claim a constant $1$ is in $\mathbb{N}(\sigma_1)$. Indeed, 
	given a partition ${\cal P}_n=\{a=x_0<x_1< \cdots <x_n=b\}$ of interval $[a, b]$, let
	\begin{eqnarray}
	\label{hatfuns}
	h_0(x) &=& \begin{cases} \displaystyle {x-x_1\over x_0-x_1} & x\in [x_0,x_1] \cr
	0 & x\in [x_1,x_n] \cr\end{cases},\\
	h_i(x) &=& \begin{cases}  {x-x_{i-1}\over x_i-x_{i-1}} & x\in [x_{i-1},x_i]\cr
	{x-x_{i+1}\over x_i-x_{i+1}} & x\in [x_i,x_{i+1}]\cr
	0 & x\not\in [x_{i-1},x_{i+1}] \cr \end{cases},
	\qquad i=1,\cdots,n-1 \\
	h_n(x) &=& \begin{cases} \displaystyle {x-x_{n-1}\over x_n-x_{n-1}} & x\in [x_{n-1},x_n]\cr
	0 & x\in [x_0,x_{n-1}] \cr \end{cases}.
	\end{eqnarray}
	be a set of piecewise linear spline functions over ${\cal P}_n$. Then we know 
	$S^0_1({\cal P}_n)=\hbox{span}\{ h_i, i=0, \cdots, n\}$
	is a linear spline space.  It is well-known that $1= \displaystyle \sum_{i=0}^n h_i(x)$. 
	Now we note the following formula:  
	\begin{equation}
	\label{mjlai}
	h_i(x) =\frac{ \sigma_1(x-x_{i-1})}{(x_i-x_{i-1})} 
	+w_i\sigma_1(x-x_i) + \frac{\sigma_1(x-x_{i+1})}{(x_{i+1}-x_{i})},
	\end{equation}
	where $w_i=-1/(x_i-x_{i-1})- 1/(x_{i+1}-x_i)$.  It follows that 
	any spline function in $S^0_1({\cal P}_n)$ can be written in terms of ReLU functions. In particular, 
	we can write $1= \sum_{i=0}^n h_i(x) = \sum_{i=0}^n c^0_i \sigma_1(x-x_i)$ by using (\ref{mjlai}).  
	
	Hence, for any linear polynomial $s(\bfx) =a + \sum_{j=1}^d c_j x_j$, we have 
	$$
	s(\bfx) = a\sum_{i=0}^n c^0_i \sigma_1(x-x_i) + \sum_{i=1}^{d} c_i \sigma_1(\bfe_j \cdot \bfx) 
	\in \mathbb{N}(\sigma_1).
	$$
	This completes the proof. 
\end{proof}

The above result shows that any linear spline is in the ReLU neual networks. Also, any ReLU neual network in $\mathbb{R}^1$ is a linear spline.   
We are now ready to prove Theorem~\ref{mjlai2019}. We begin with the standard modulus of continuity. 
For any continuous function $g\in C[0, d])$, we define the modulus of continuity of $g$ by 
\begin{equation}
\label{modulusofsmoothness}
\omega(g, h)= \max_{ x\in [0,d]\atop 0< t\le h} |g(x+t)- g(x)|
\end{equation}
for any $h>0$.  
To prove the result in Theorem~\ref{mjlai2019}, we need to recall some basic properties 
of linear splines (cf. \cite{S07}). The following result was established in \cite{P81}. 
\begin{lemma}
\label{Powell}
For any function in $f\in C[a,b]$, there exists a linear spline $S_f\in S^0_1(\triangle)$ such that 
\begin{equation}
\|f- S_f\|_{\infty, [a,b]} \le \omega(f, |\triangle|)
\end{equation}
where  $S^0_1(\triangle)$ is the space of all continuous linear splines over the partition
$\triangle=\{a=t_0<t_1<\cdots <t_n=b\}$ with $|\triangle|=\max_i |t_i-t_{i-1}|$.  
\end{lemma}
In order to know the rate of convergence, we need to introduce the class of function of  bounded variation. 
We say a function $f$ is of bounded variation over $[a, b]$ if 
$$
\sup_{\forall a=x_0<x_1< \cdots <x_n=b} \sum_{i=1}^n |f(x_i)-f(x_{i-1})|  < \infty
$$
We let $V_a^b(f)$ be the value above when $f$ is of bounded variation. The following result is well known (cf. \cite{S07})
\begin{lemma}
\label{key0}
Suppose that $f$ is of bounded variation over $[a, b]$.  
For any $n\ge 1$,  there exists a partition 
$\triangle$ with $n$ knots such that  
$$
\hbox{dist}(f, S^0_1(\triangle))_\infty=\inf_{s\in S^0_1(\triangle)}
\|f -s \|_{\infty} \le \frac{V_a^b(f)}{n+1}. 
$$
\end{lemma}

Let $L=\{ f\in C([0,1]^d):  |f(\bfx)-f(\bfy)|\le L_f |\bfx- \bfy|, \forall \bfx, \bfy\in [0, 1]^d\}$ be the class of 
Lipschitz continuous functions. We can further establish  
\begin{lemma}  
\label{key}
Suppose that $f$ is Lipschitz continuous over $[a, b]$ with Lipschitz constant $L_f$.  
For any $n\ge 1$,  there exists a partition 
$\triangle$ with $n$ interior knots such that  
$$
\hbox{dist}(f, S^0_1(\triangle))_\infty\le \frac{L_f(b-a)}{2(n+1)}. 
$$
\end{lemma}
\begin{proof}
We use a linear interpolatory spline $S_f$. Then for $x\in [x_i, x_{i+1}]$, 
\begin{eqnarray*}
f(x)- S_f(x) &=& f(x)- f(x_i)\frac{x_{i+1}- x}{x_{i+1}-x_i}- f(x_{i+1}) \frac{x-x_i}{x_{i+1}-x_i}.
\cr
&=& \frac{(x_{i+1}- x)(f(x)- f(x_i))}{x_{i+1}-x_i}+ 
\frac{(x-x_i)(f(x)- f(x_{i+1}))}{x_{i+1}-x_i} \cr
&\le& L_f \frac{(x- x_i)(x_{i+1}- x)}{x_{i+1}-x_i} + L_f\frac{(x_{i+1}-x) 
(x-x_i)}{x_{i+1}-x_i} \le  \frac{L_f}{2}(x_{i+1}-x_i).
\end{eqnarray*}
Hence, $|f(x)- S_f(x)|  \le L_f(b-a)/(2(n+1))$ if $x_{i+1}-x_i=(b-a)/(n+1)$.  
This completes the proof. 
\end{proof}

Furthermore, if $f$ is Lipschitz continuous, 
so is the linear interpolatory spline $S_f$. In fact, we have 
\begin{equation}
\label{key2}
|S_f(x)-S_f(y)|\le L_f |x-y|.
\end{equation}
We are now ready to prove one of our main results in this paper. 

\begin{proof} [Proof of   Theorems~\ref{mjlai2019}]
Since $\phi_q$ are univariate increasing functions mapping from $[0, 1]$ to $[0, 1]$, they are
bounded variation with $V_0^1(\phi_q)\le 1$.  By Lemma~\ref{key0}, 
there are linear spline functions $L_q$ 
such that $|L_q(t)- \phi_q(t)|\le 1/(n+1)$ for $q=0, \cdots, 2d$. 

For the outer function $g_f$ associated with $f$, when $g_f$ is Lipschitz continuous, there is a linear 
spline $S_g$ with $dn$ distinct interior knots over $[0, d]$ such that 
$$
\sup_{t\in [0, d]} |g_f(t)- S_g(t)| \le \frac{dC_g}{2(nd+1)}\le \frac{C_g}{2n},
$$ 
where $C_g$ is  the Lipschitz constant of $g$ by using Lemma~\ref{key}. 
Now we first have
\begin{eqnarray*}
|f(\bfx)- \sum_{q=0}^{2d} S_g(\sum_{i=1}^d\lambda_i \phi_q(x_i))| 
\le \sum_{q=0}^{2d} |g(\sum_{i=1}^d\lambda_i \phi_q(x_i)) - S_g(\sum_{i=1}^d\lambda_i \phi_q(x_i))| 
\le \frac{(2d+1)C_g }{2n}.
%\le (2d+1)C_gd/(2(n+1)).
\end{eqnarray*}

Next since $g$ is Lipschitz continuous, so is $S_g$.  Thus, by (\ref{key2}) and Lemma~\ref{key0}, we have
\begin{eqnarray*}
|S_g(\sum_{i=1}^d\lambda_i \phi_q(x_i))- S_g(\sum_{i=1}^d\lambda_i 
L_q(x_i))|
\le C_g  \sum_{i=1}^d\lambda_i |\phi_q(x_i) -L_q(x_i)| \le \frac{d\cdot C_g }{(n+1)}.
\end{eqnarray*}

Let us put the above estimates together to have
$$
|f(\bfx)- \sum_{q=0}^{2d} S_g(\sum_{i=1}^d\lambda_i L_q(x_i))|\le
\frac{(2d+1)C_g}{2n} + \frac{(2d+1)d C_g }{(n+1)}\le \frac{(2d+1)^2 C_g }{n}.
$$
The conclusion of Theorem~\ref{mjlai2019} follows. 
\end{proof}

\subsection{Functions Beyond the KL Class} 
As the outer function $g$ may not be Lipschitz continuous, we next consider a class of 
functions which is of H\"older continuity. 
Letting $\alpha\in (0, 1]$, we say $g$ is in $C^{0,\alpha}$ if 
\begin{equation}
\label{Lipalpha}
  \sup_{x,y\in [0, d]}  \frac{|g(x)- g(y)|}{|x- y|^\alpha} \le L_\alpha(g) <\infty. 
\end{equation}

Using such a continuous function $g$, we can define a multivariate continuous  function $f$ by using the formula in  Theorem~\ref{kolmogorov}.  
In addition, we can extend the argument for the proof of Theorem~\ref{mjlai2019} 
to the setting of Kolmogorov-H\"older continuous functions and present the convergence 
in terms of Kolmogorov-modulus of smoothness. 
Let us extend the analysis of the proof of Lemma~\ref{key} to have

\begin{lemma}
\label{nkey}
Suppose that $g$ is  H\"older continuous over $[0, d]$, say $g\in C^{0,\alpha}$ with 
$L_\alpha(g)$ for some $\alpha\in (0, 1]$.  For any $n\ge 1$,  there exists a partition 
$\triangle$ with $n$ interior knots such that  
$$
\hbox{dist}(g, S^0_1(\triangle))_\infty\le \frac{L_\alpha(g) d^\alpha}{2(n+1)^\alpha}. 
$$
\end{lemma}

%Similar to the proof of Theorem~\ref{mjlai2019}, we have
Similarly, we can define a class of functions which is Kolmogorov-H\"older (KH) continuous in the sense
that $K$-outer function $g$ is H\"older continuity $\alpha\in (0, 1)$. For each univariate $g$ in 
$C^{0,\alpha}([0, d])$, we define $f$ using the KST formula (\ref{repres}). Then we have a new class of
continuous functions which will satisfy (\ref{subcurse_holder}). The proof is a straightforward generalization of the one for Theorem~\ref{mjlai2019}, we leave it to the interested readers.

\begin{theorem} 
\label{mjlai2019n}
For each continuous function $f\in C([0, 1]^d)$,  
let $g$ be the outer function associated with $f$.  Suppose that $g$ is in 
$C^{0,\alpha}([0, d])$ for some $\alpha\in (0, 1]$. Then 
\begin{equation}
\label{subcurse_holder}
\inf_{ s\in {\cal K}_{n, n}(\sigma_1) } \|f - s\|_{C([0,1]^d) } 
\le \frac{(2d+1)^2 L_{\alpha}(g) }{n^\alpha}.
\end{equation}
\end{theorem}
% \begin{proof} The proof is a straightforward generalization 
% of the one for Theorem~\ref{mjlai2019}. The detail is omitted. \end{proof}

Finally, in this section, we study the Kolmogorov-modulus of continuity.  
For any continuous function $f\in C([0,1]^d)$, let $g_f$ be the outer function of $f$ based on
the KST. Then we use $\omega(g_f,h)$ which is called the Kolmogorov-modulus of continuity of $f$ to measure the 
smoothness of $g_f$. 
Due to the uniform continuity of $g_f$, we have  linear spline 
$S_{g_f}$ over an equally-spaced knot sequence such that 
\begin{equation}
| g_f(t) -S_{g_f}(t)|\le \omega(g_f,h), \quad \forall t\in [0, d] 
\end{equation}  
for any $h>0$, e.g. $h=1/n$ for a positive integer $n$.  It follows that 
\begin{equation}
\label{newkey1}
| g_f(\sum_{i=1}^d \lambda_i \phi_q(x_i))- S_{g_f}(\sum_{i=1}^d \lambda_i \phi_q(x_i)|\le \omega(g_f,h),
\end{equation}
for any $(x_1,\cdots, x_d)\in [0,1]^d$.   
Since $\phi_q, q=0, \cdots, 2d$ are monotonically increasing, we use Lemma~\ref{key0} to have linear splines $L_q$ such that $|L_q(t)- \phi_q(t)|\le h$ since $V_0^1(\phi_q)\le 1$.  We now estimate
\begin{equation}
\label{newkey2}
 | S_{g_f}(\sum_{i=1}^d \lambda_i \phi_q(x_i))- S_{g_f}(\sum_{i=1}^d \lambda_i L_q(x_i))|
\end{equation}
for $q=0, \cdots, 2d$. Note that 
$$
 | \sum_{i=1}^d \lambda_i \phi_q(x_i) - \sum_{i=1}^d \lambda_i L_q(x_i)| \le \sum_{i=1}^d |\phi_q(x_i)- L_q(x_i)| \le dh.
$$
The difference of the above two points in $[0, d]$ is separated by at most $d$ subintervals with length $h$ and hence,  we will have   
\begin{equation}
\label{newkey3}
| S_{g_f}(\sum_{i=1}^d \lambda_i \phi_q(x_i))- S_{g_f}(\sum_{i=1}^d \lambda_i L_q(x_i))| 
\le  2d\cdot\omega(g_f,h)  
\end{equation}
since $S_{g_f}$ is a linear interpolatory spline of $g_f$.  It follows that 
\begin{eqnarray*}
&&|f(x_1, \cdots, x_n) - \sum_{q=0}^{2d} S_{g_f}(\sum_{i=1}^d \lambda_i L_q(x_i))| \cr
&\le & \sum_{q=0}^{2d} | g_f(\sum_{i=1}^d \lambda_i \phi_q(x_i))- S_{g_f}(\sum_{i=1}^d \lambda_i \phi_q(x_i))| + 
\sum_{q=0}^{2d}|S_{g_f}(\sum_{i=1}^d \lambda_i \phi_q(x_i))- S_{g_f}(\sum_{i=1}^d \lambda_i L_q(x_i))|\cr
&\le & (2d+1)\omega(g_f,h)+ (2d+1)2d\cdot\omega(g_f,h).
\end{eqnarray*}
Therefore, we conclude the following theorem.
\begin{theorem} 
	\label{mjlai2022}
	For any continuous function $f\in C[0,1]^d$, let $g_f$ be the outer function associated with $f$.  Then 
	\begin{equation}
	\label{subcurse2}
	\inf_{ s\in \mathcal{K}_{n, n}(\sigma_1) } 
	\|f - s\|_{C([0,1]^d) } \le (2d+1)^2\omega(g_f,1/n).
	\end{equation}  
\end{theorem}

%We now give some  comment on the significance of the result above. First of all,   

%\section{Numerical Experimental Results}
\section{KB-splines and LKB-splines}
\label{sec3}
However, it is not easy to see if the outer function $g_f$ is Lipschitz continuous  when given a continuous functions $f$. To do so we have
to compute $g_f$ from $f$ first.  To this end, we implemented Lorentz's constructive proof of KST in MATLAB by following the steps in pages $168 - 174$ in \cite{L66}. See \cite{B08} for another implementation based on Maple and MATLAB. 
We noticed that the curve $g_f$ behaviors very badly for many smooth functions $f$.  
Even if $f$ is a linear polynomial in the 2-dimensional space, the outer function $g$ still behaviors very widey 
although we can use K-network with two hidden layers to approximate this linear polynomial $f$ arbitrarily well  in theory.  This may be a big hurdle to prevent researchers in \cite{Girosi1989}, \cite{H87}, 
\cite{K92}, \cite{IP03}, \cite{B08}, and etc. from successful applications based on Kolmogorov spline network.    
We circumvent the difficulty of having such a wildly behaved outer function $g$ by introducing KB-splines and the denoised counterpart LKB-splines in this section. 
In addition, we will explain how to use them to well approximate high dimensional functions in a later
section..  

First of all, we note that the implementation of these $\phi_q, q=0, \cdots, 2d$ is not easy. 
Numerical $\phi_q$'s are not accurate enough. Indeed, 
letting $z_q(x_1,\cdots, x_d)= \sum_{i=1}^d \lambda_i \phi_q(x_i), $
Consider the transform: 
\begin{equation}
\label{peano}
T(x_1,\cdots, x_d)=(z_0, z_1, \cdots, z_{2d}) 
\end{equation} 
which maps from $[0, 1]^d$ to  $\mathbb{R}^{2d+1}$. Let $Z=\{T(x_1,\cdots, x_d), (x_1,\cdots, 
x_d)\in [0, 1]^d\}$ be the image of $T([0,1]^d)\subset \mathbb{R}^{2d+1}$. It is easy to see that
the image is closed. The theory in \cite{L66} explains that the map $T$ is one-to-one  
and continuous. As the dimension of
$Z$ is much larger than $d$, the map $T$ is like a well-known Peano curve which maps from $[0, 1]$
to $[0,1]^2$ and hence, the implementation of $T$, i.e., the implementation of $\phi_q$'s 
is not possible to be accurate.   
However, we are able to compute these $\phi_q$ and decompose $g$ such that the reconstruction 
of  constant function is exact. Let us present two  examples to show that 
our numerical implementation is reasonable. 
For convenience, let us use images as  2D functions and compute their outer functions $g$ 
and then reconstruct the images back.  
In Figure~\ref{im1}, we can see that the reconstruction
is very good visually although the outer functions $g$ are oscillating very much. It is worthwhile to note that such reconstruction results have also been reported in \cite{B08}.
Certainly, these images are not continuous functions and
hence we do not expect that $g$ to be Lipschitz continuous. But these reconstructed images serves as a ``proof" that our computational code 
works numerically.  
%It is known that when $f$ is a constant, $g=f/(d+1)$. See \cite{B08}. 

\begin{figure}[thpb]
    \centering
	\begin{tabular}{ccc}
		\small{Original Image} & \small{Reconstructed Image} & \small{Associated Function g} \\
		\includegraphics[width = 0.25\textwidth, height=0.25\textwidth]{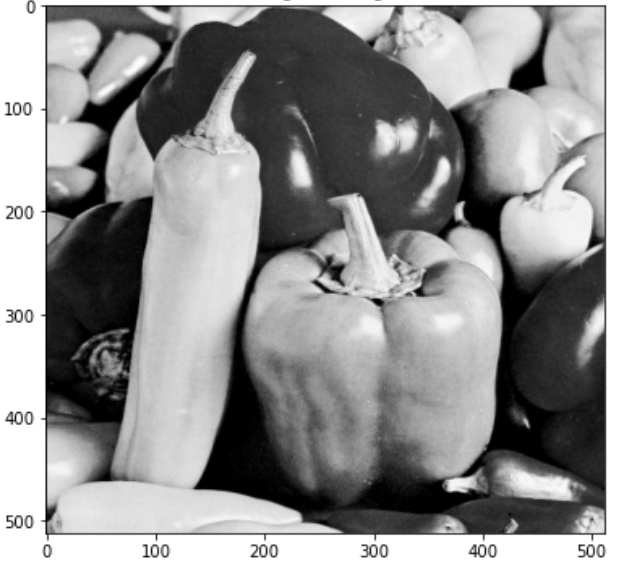} &  \includegraphics[width = 0.25\textwidth, height=0.25\textwidth]{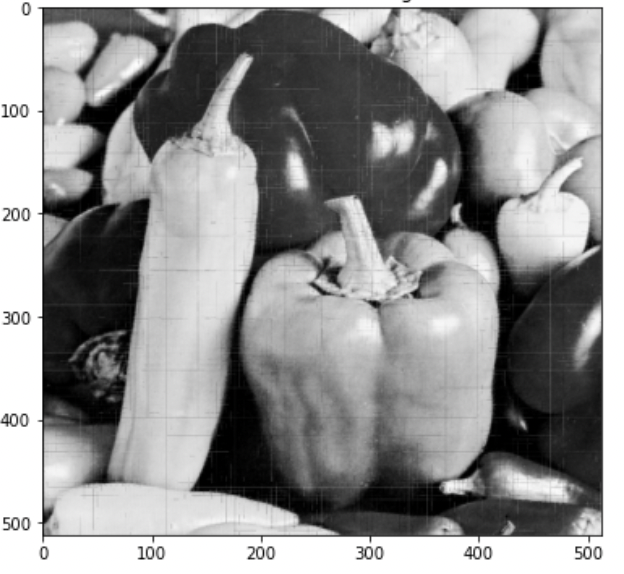} & \includegraphics[width = 0.25\textwidth, height=0.25\textwidth]{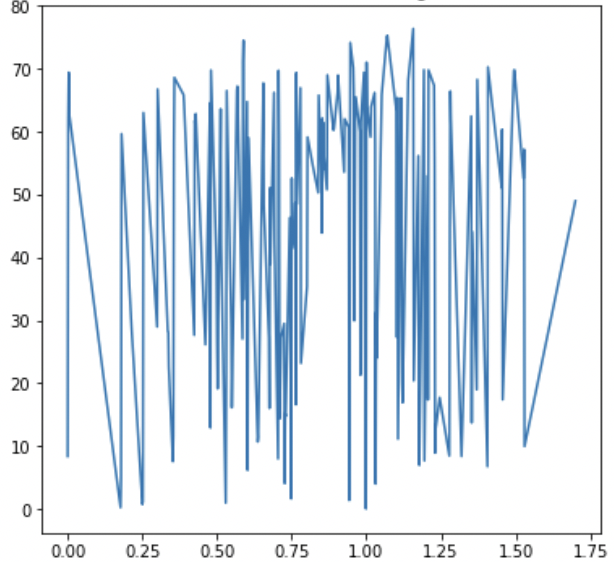} \\
  		\includegraphics[width = 0.25\textwidth, height=0.25\textwidth]{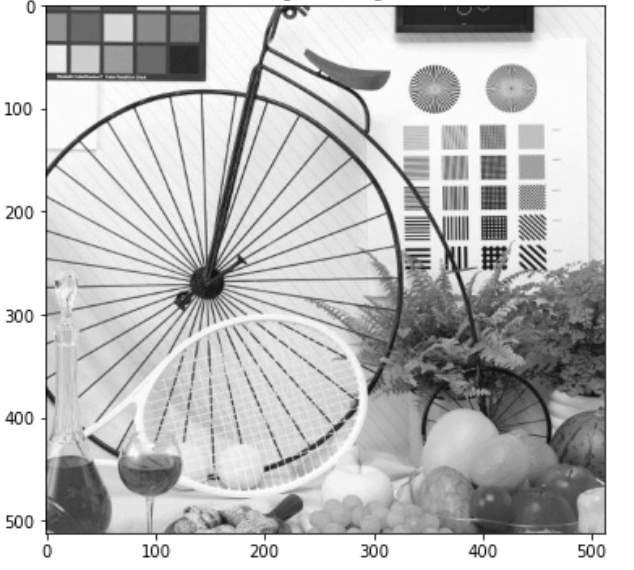} &  \includegraphics[width = 0.25\textwidth, height=0.25\textwidth]{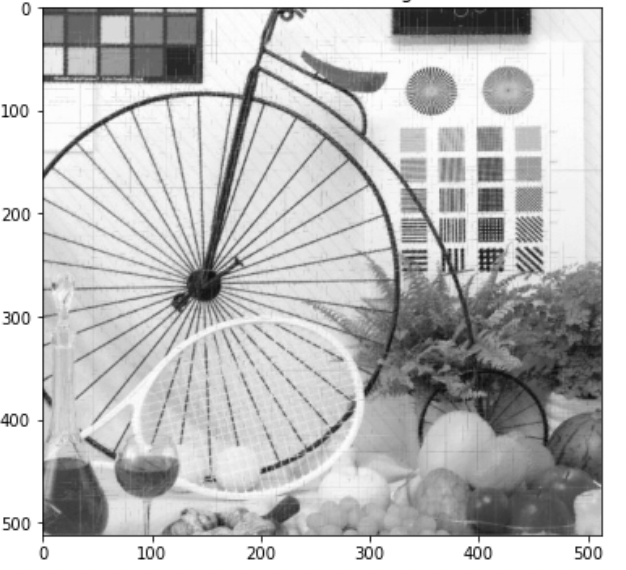} & \includegraphics[width = 0.25\textwidth, height=0.25\textwidth]{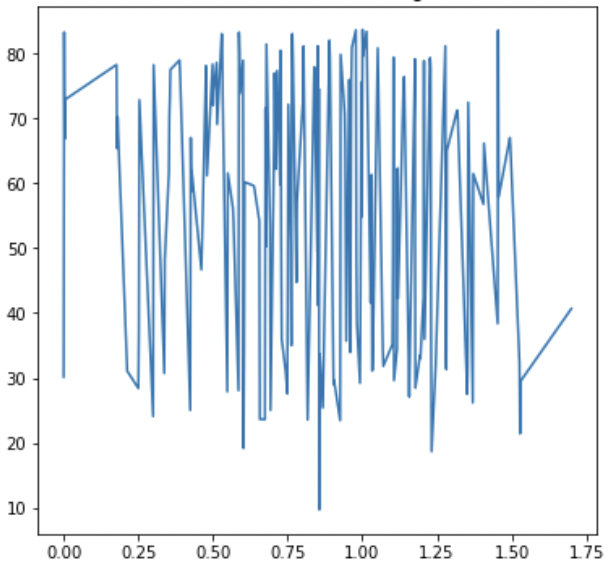} 
	\end{tabular}
	\caption{Original image (left column), reconstructed image (middle column), and associated outer function $g$ (right column)\label{im1}}
\end{figure}

% \begin{figure}[thpb]
% 	\begin{tabular}{ccc}
% 		\small{Original Image} & \small{Reconstructed Image} & \small{Associated Function g} \\
% 		\includegraphics[width = 0.3\textwidth, height=0.3\textwidth]{bike_original.png} &  \includegraphics[width = 0.3\textwidth, height=0.3\textwidth]{bike_reconstructed_k=4.png} & \includegraphics[width = 0.3\textwidth, height=0.3\textwidth]{g_bike.png} 
% 	\end{tabular}
% 	\caption{Original image (left) and reconstructed image (middle) based the K-outer function $g$ (right)\label{im2}}
% \end{figure}

Next we present a few examples of smooth functions whose outer functions may not be Lipschitz 
continuous in Figure~\ref{im3}. Note that the reconstructed functions are very noisy, in fact 
they are too noisy to believe that the implementation of the KST can be useful.  In order to
see that these noisy functions are indeed the original functions, we applied  
a penalized least squares method based on  bivariate spline method  
%(cf. \cite{LW13} and \cite{WWL20}). 
(to be explained later in the paper).  
That is, after denoising, the reconstructed functions are very close to the exact original functions
as shown in Figure~\ref{im3}. That is, the denoising method is successful which motivates us to 
adopt this approach to approximate any continuous functions. 

\begin{figure}
\centering
\begin{tabular}{cc} 
    \includegraphics[width = 0.48\textwidth, height=0.25\textwidth]{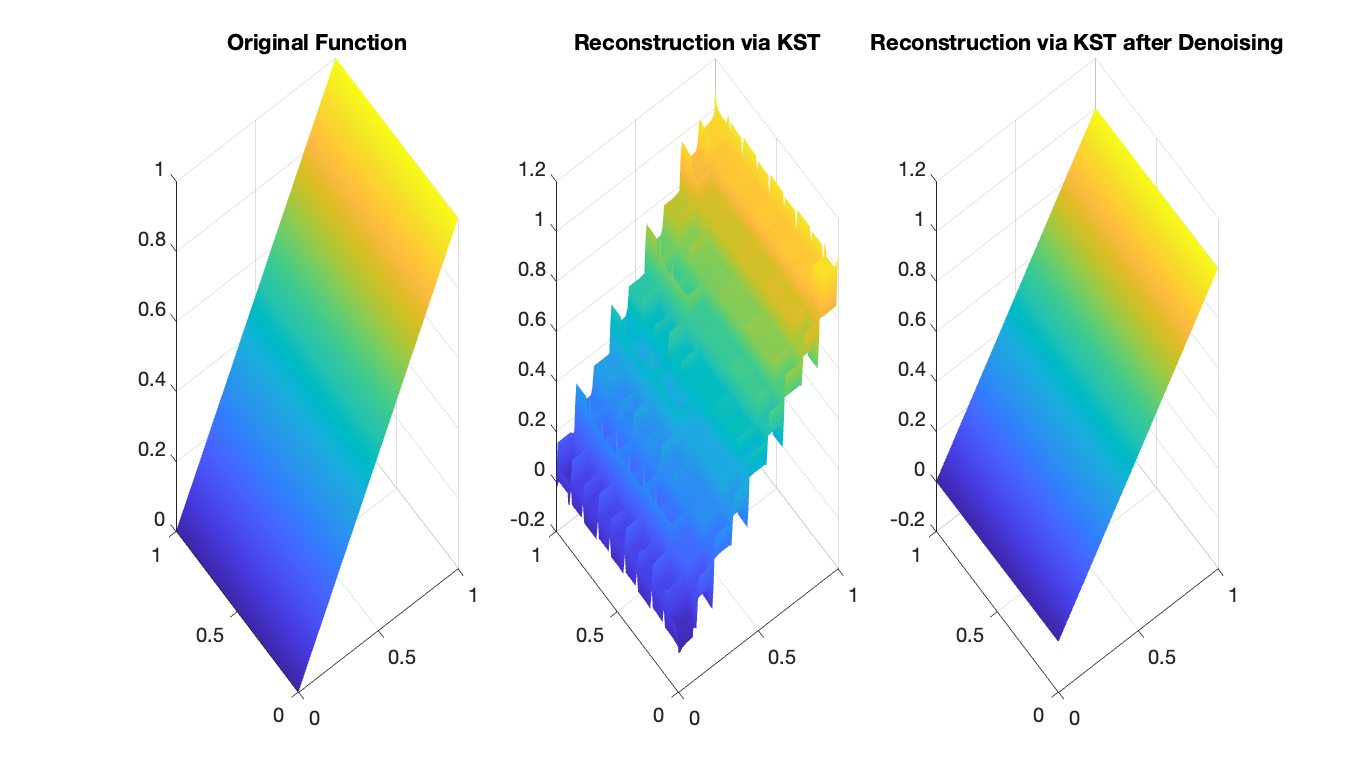} &
    \includegraphics[width = 0.48\textwidth, height=0.25\textwidth]{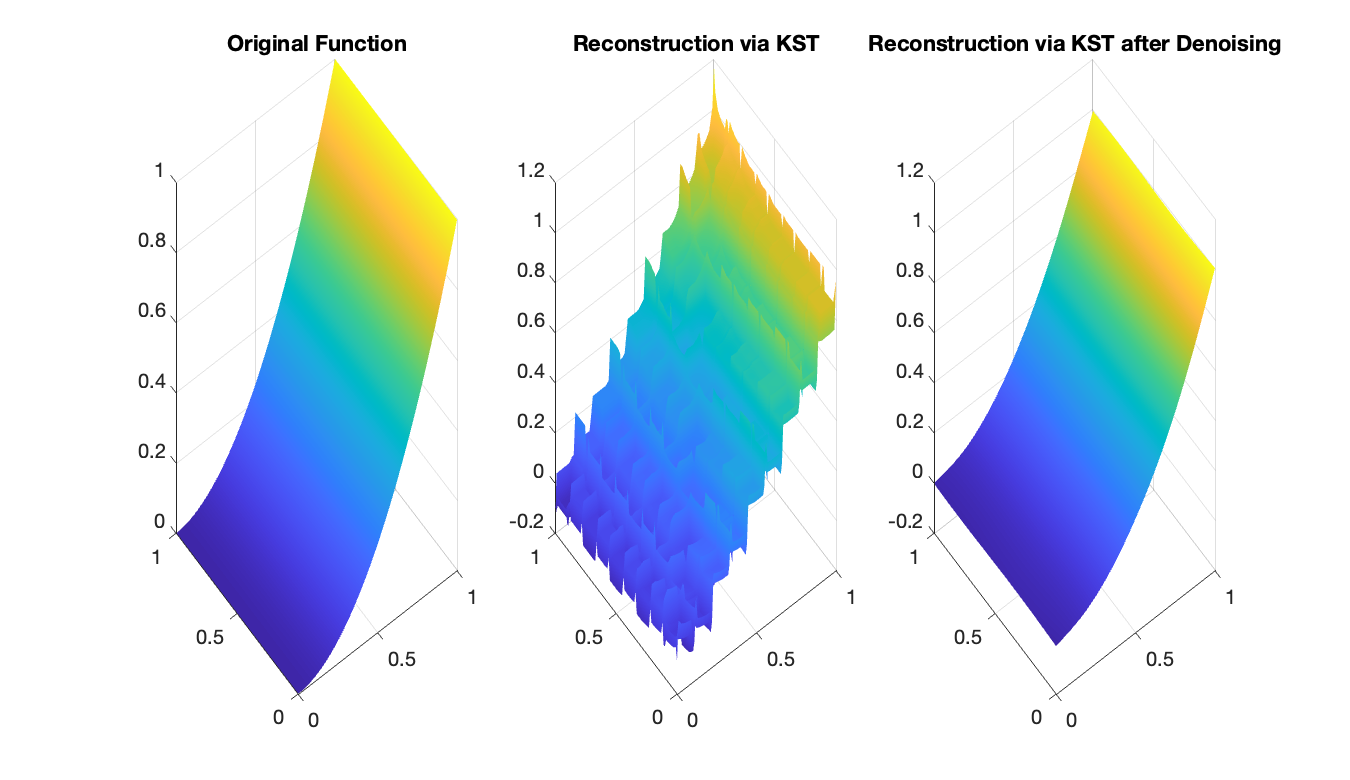} 
    \\
    \includegraphics[width = 0.48\textwidth, height=0.25\textwidth]{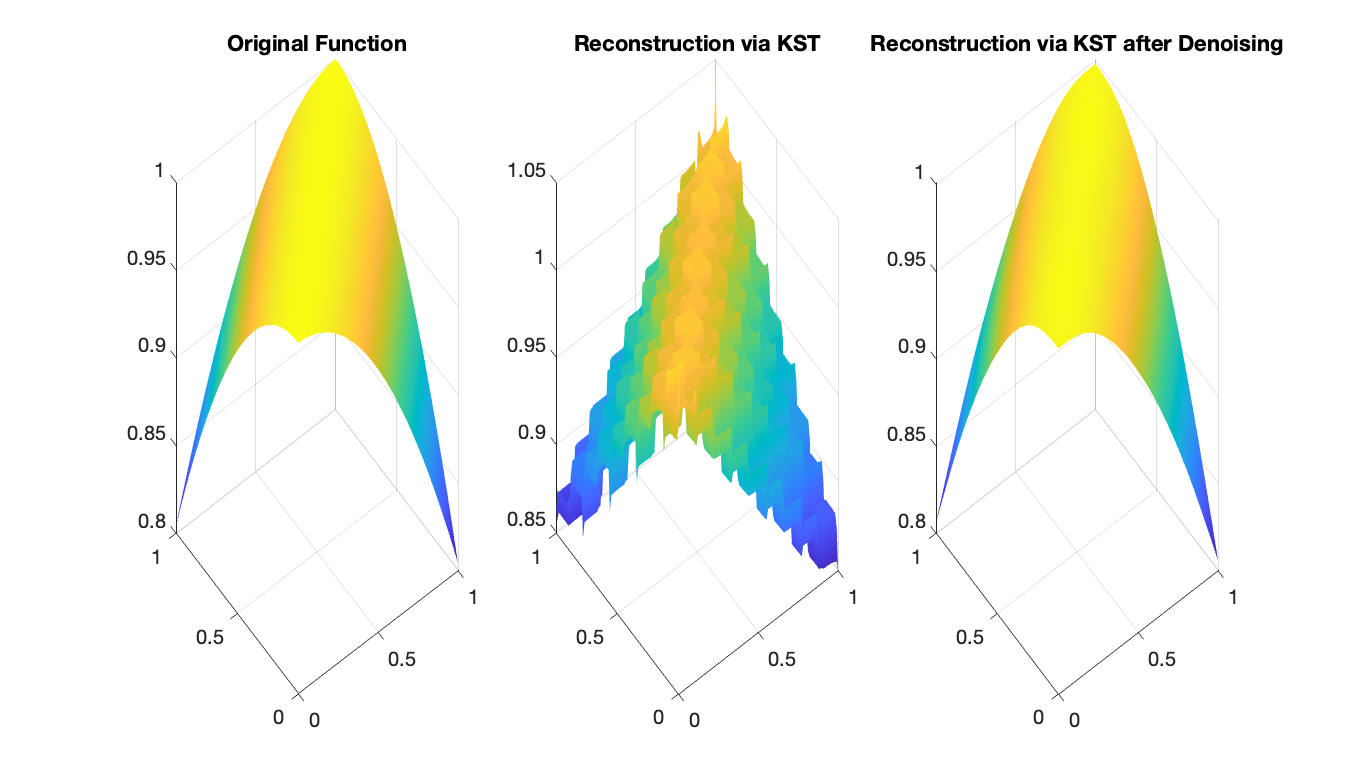} &
    \includegraphics[width = 0.48\textwidth, height=0.25\textwidth]{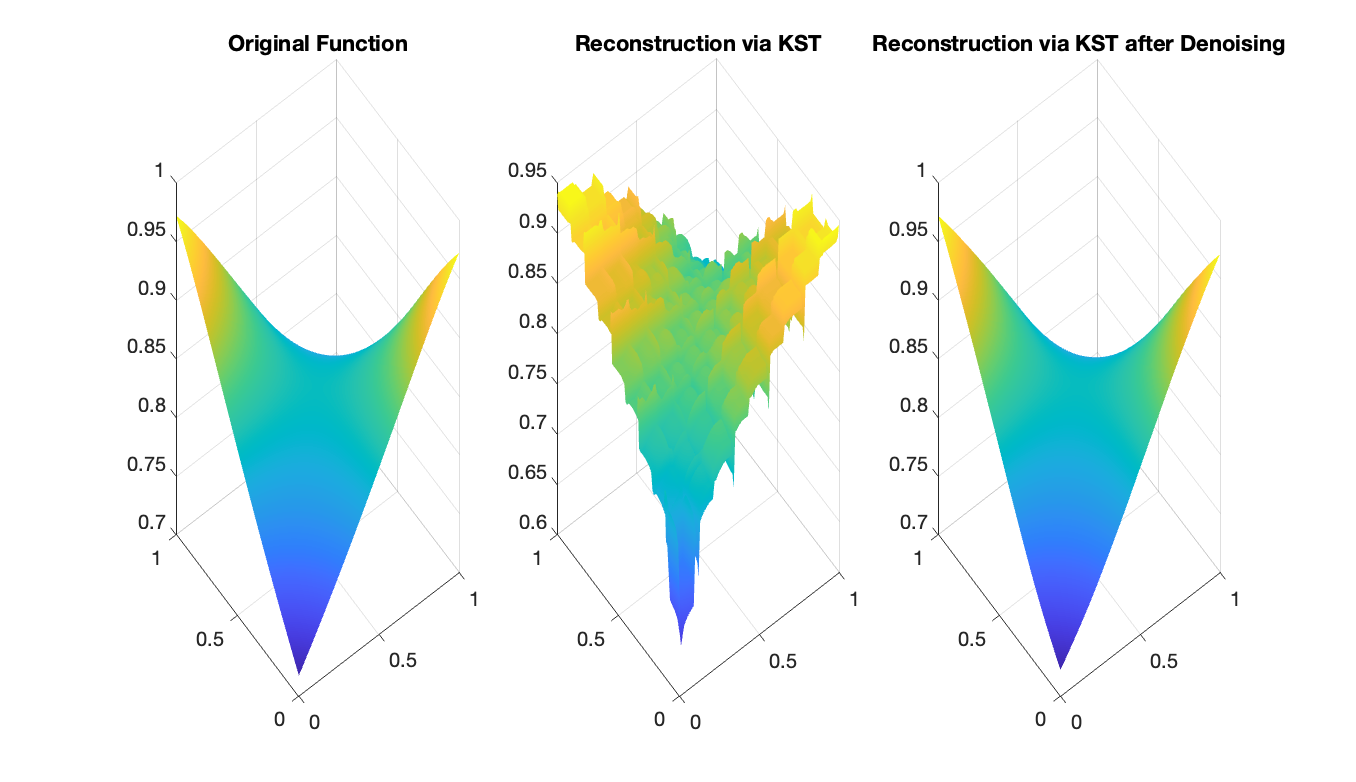}
\end{tabular}
\caption{Top left: reconstruction of $f(x,y)=x$. Top right: reconstruction of $f(x,y)=x^2$. Bottom left: reconstruction of $f(x,y)=\cos(2(x-y)/\pi)$.  Bottom right: reconstruction of $f(x,y)=\sin(1/(1+(x-0.5)(y-0.5)))$. \label{im3}}
\end{figure}

\subsection{KB-splines}
To this end, we first use standard uniform B-splines to form some subclasses of KL continuous functions. 
Let $\triangle_n=\{ 0= t_1<t_2<\cdots <t_{dn}<d\}$ 
be a uniform partition of interval $[0, d]$ 
and let $b_{n,i}(t)= B_k(t-t_i), i=1, \cdots, dn$ 
be the standard B-splines of degree $k$ with $k\ge 1$. 
For simplicity, we only explain our approach based on linear B-splines for the theoretical 
aspect while using other B-splines (e.g. cubic B-splines) for the numerical experiments. We define KB-splines by 
\begin{equation}
\label{KB}
KB_{n,j}(x_1,\cdots, x_d) = \sum_{q=0}^{2d} b_{n,j}\left(\sum_{i=1}^d\lambda_i\phi_q(x_i)\right), 
j=1, \cdots, dn.
\end{equation}
It is easy to see that each of these KB-splines defined above is nonnegative.  
Due to the property of B-splines: $\sum_{i=1}^{dn} b_{n,i}(t)=1$ for all $t\in [0, d]$, we have the following property of KB-splines:  
\begin{theorem}
\label{constantpreserving}
We have	$\sum_{i=1}^{dn} KB_{n,i}(x_1,\cdots, x_d)=1$ and hence, $0\le KB_{n,i} \le 1$.  
\end{theorem}
\begin{proof}
The proof is immediate by using the fact $\sum_{i=1}^{dn} b_{n,i}(t)=1$ for all $t\in [0, d]$.
\end{proof}
\begin{remark}
The property in Theorem~\ref{constantpreserving} is called the partition of unit which makes the computation stable.  
We note that a few of these $dn$ KB-splines will be zero since $0<\lambda_i\leq 1$ and $\min\{\lambda_i, i=1, \cdots, d\}<1$. The number of zero KB-splines is dependent on the choice of $\lambda_i$, $i=1,\cdots, d$.  
\end{remark}
Another important result is that these KB-splines are linearly independent.
\begin{theorem}
\label{independence}
The nonzero KB-splines $\{KB_{n,j}\neq 0, j=1, \cdots, dn\}$ are linearly independent.
\end{theorem}
\begin{proof}
Suppose there are $c_j$, $j=1,2,\cdots,dn$ such that $\sum_{j=1}^{dn}c_j KB_{n,j}(x_1,\cdots,x_d)=0$ for 
all $(x_1, \cdots, x_d)\in [0, 1]^d$. Then 
we want to show $c_j=0$ for all $j=1,2,\cdots, dn$.
Let us focus on the case $d=2$ as the proof for general case $d$ is similar.  
Suppose $n>0$ is a fixed integer and we use the notation $z_q=\sum_{i=1}^2\lambda_i\phi_q(x_i)$ as above. 
Then based on the graphs of $\phi_q$ in Figure~\ref{phis}, we can choose $x_1=\delta$ and $x_2=0$ with $0<\delta\leq 1$ 
small enough such that $KB_{n,j}(x_1,0)=\sum_{q=0}^{4} b_j(z_q(\delta,0))=0$ for all 
$j=3,4,\cdots, 2n$. Therefore in order to show the linear independence of $ KB_{n,j}, j=1,2,\cdots, 2n$, it is suffices to show $\sum_{j=1}^{2} c_j KB_{n,j}(x_1,x_2)=0$ implies $c_1=c_2=0$. 
Let us confine $x_1\in [0,\delta]$ and $x_2=0$. Then we have 
\begin{align*}
0 &= c_1 KB_{n,1}(x_1,x_2) + c_2 KB_{n,2}(x_1,x_2) = c_1(\sum_{q=0}^{4}b_1(z_q))+c_2(\sum_{q=0}^{4}b_2(z_q)) \\
&=c_1(\sum_{q=0}^{4}b_1(z_q))+c_2(5-\sum_{q=0}^{4}b_1(z_q)) =(c_1-c_2)(\sum_{q=0}^{4}b_1(z_q))+5c_2,
\end{align*}
where we have used the fact that $b_1(x)+b_2(x)=1$ over $[0,1/n]$.   
Since $c_2$ is constant, and $\sum_{q=0}^{4}b_1(z_q)$ is not constant when $x_1$ varies between $0$ and $\delta$, we must have $c_1=c_2$. Hence $c_2=0$ and therefore $c_1=0$. 

\begin{figure}[h]
    \centering
    \includegraphics[width=5in, height=2in ]{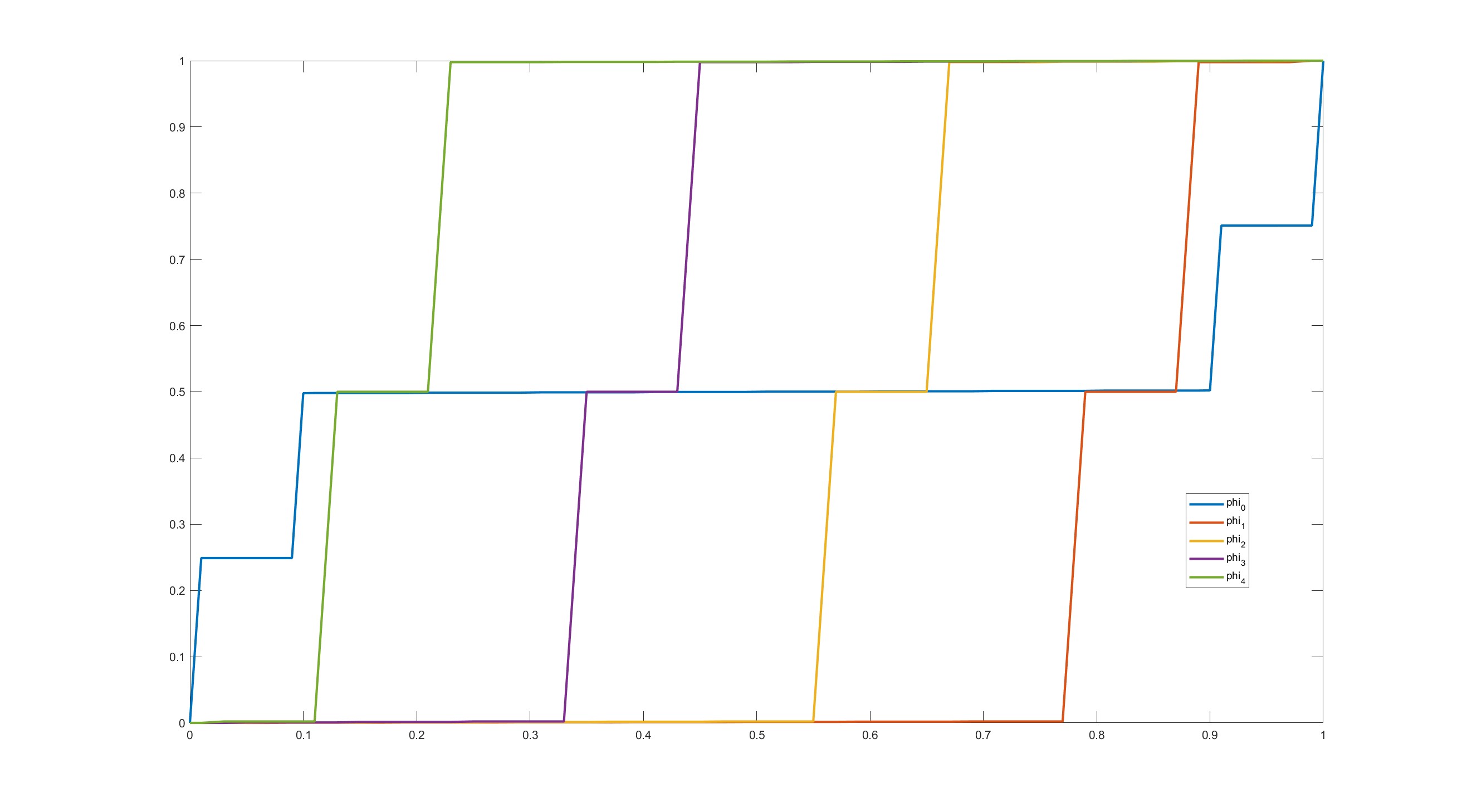}
    \caption{$\phi_q$, $q=0, 1, 2, 3, 4$, in the 2D setting. \label{phis}}
\end{figure}

In the same fashion, we can choose $\tilde{x}_1$ and $\tilde{\delta}$ such that 
$KB_{n,j}=\sum_{q=0}^{4} b_j(z_q(\tilde{x}_1,0))=0$ for all  $j=1,2, \cdots, 2n$ except for $j=k, k+1$. 
By the similar argument as above, we have $c_k=c_{k+1}=0$. 
By varying $k$ between $1$ and $2n$, we get $c_j=0$ for all $j=1, 2, \cdots, n$.
\end{proof}

Since $\hbox{span}\{b_{n,i}, i=1, \cdots, nd\}$ will be dense in $C[0, d]$ when $n\to \infty$,
we can conclude that $\hbox{span}\{KB_{n,j}, j=1,\cdots, nd\}$ will be dense in $C([0, 1]^d)$.  That is, we have 
\begin{theorem}
	\label{dense}
The KB-splines  $KB_{n,j}(x_1,\cdots, x_d), j=1, \cdots, nd$ are dense in $C([0, 1]^d)$ when $n\to \infty$ for a fixed dimension $d\ge 2$.    
\end{theorem}
\begin{proof}
For any continuous function $f\in C([0, 1]^d)$, let $g_f\in C[0, d])$ be the outer function
of $f$. For any $\epsilon>0$, there is an integer $n>0$ and a spline $S_{g_f}\in 
\hbox{span}\{b_{n,i}, i=1, \cdots, dn\}$ such that 
\begin{equation} \label{keyestimate}
\|g_f(t)- S_{g_f}(t)\|_{\infty} \le \epsilon/(2d+1)   
\end{equation}
for all $t$. Note that one can find $S_{g_f}$ by partitioning the interval $[\min(g_f), \max(g_f)]$ into subintervals with length $\epsilon/(2d+1)$ and
find the knots over $[0, d]$ and then add more knots in addition to modify the existing one to obtain a uniform knot sequence.  Then $S_{g_f}$ is a 
linear interpolatory spline of $g_f$ based on the uniform knot sequence.  

Writing $S_{g_f}(t) = \sum_{i=1}^{dn} c_i(f) b_{n,i}(t)$, we have 
\begin{eqnarray*}
\label{estimate1}
&&|f(x_1,\cdots, x_d) - \sum_{i=1}^{dn} c_i(f) KB_{n,i} (x_1,\cdots,x_d)|\cr
&=& |\sum_{q=0}^{2d} g(z_q(x_1,\cdots, x_d))- \sum_{i=1}^{dn} c_i(f) \sum_{q=0}^{2d}  b_{n,i}(z_q(x_1,\cdots,x_d))|\cr
&\le & \sum_{q=0}^{2d}|g(z_q(x_1,\cdots,x_d))-S_{g_f}(z_q(x_1,\cdots,x_d))| \le  (2d+1)\epsilon/(2d+1)=\epsilon.
\end{eqnarray*}
This completes the proof. 
\end{proof}

\begin{corollary}
\label{newmain3}
Suppose that $f\in C([0, 1]^d)$ is Kolmogorov-Lipschitz continuous, that is, the outer function $g_f\in C([0, d])$ is Lipschitz with Lipschitz constant $L$. Then there exists a KB spline
$S_n\in  \hbox{span}\{KB_{n,j}(x_1,\cdots, x_d), j=1, \cdots, nd\}$ such that 
\begin{equation}
|f(x_1,\cdots, x_d) - \sum_{i=1}^{dn} c_i(f) KB_{n,i} (x_1,\cdots,x_d)| \le (2+1/d)L/n.
\end{equation}
\end{corollary}
\begin{proof}
In the proof of Theorem~\ref{dense},  we used the key estimate (\ref{keyestimate}).  When $g_f$ is Lipschitz, we have (cf. Theorem 20.2 in \cite{P81})
\begin{equation}
\label{keyestimate2}
\|g_f(t)- S_{g_f}(t)\|_{\infty} \le \omega(g_f, 1/(nd))\le L/(nd).
\end{equation}
The rest of the proof is the same as the one in the proof of Theorem~\ref{dense}.  
\end{proof}

Note that the computation of $c_i(f)$'s is not easy as we do not know $g_f$.  We shall explain a computational method to approximate $f$ in the following sections.

\subsection{LKB-splines}
However, in practice, the KB-splines obtained in (\ref{KB}) are very noisy due to any implementation of 
$\phi_q$'s as we have explained before that the functions $z_q, q=0, \cdots, 2d$, like 
Peano's curve. One has no way to have an accurate implementation.  Aslso, as demonstrated before, 
our denoising method can help. We shall call LKB-splines after denoising KB-splines.  

Let us explain a multivariate spline method for denoising for $d=2$ and $d=3$. In general, we can use tensor product B-splines for denoising
for any $d\ge 2$ which is the similar to what we are going to explain below. 
For convenience, let us consider $d=2$ and let $\triangle$ be a 
triangulation of $[0, 1]^2$ based on a uniform refinement of two triangles by adding a diagonal to $[0, 1]^2$.  
For any degree $D\ge 1$ and smoothness $r\ge 1$ with $r<D$, let
\begin{equation}
\label{splinespace}
S^r_D(\triangle)= \{ s\in C^r([0, 1]^2):  s|_T\in \mathbb{P}_D, T\in \triangle\}
\end{equation}
be the spline space of degree $D$ and smoothness $r$ with $D>r$.   We refer to \cite{LS07} for a theoretical detail and \cite{ALW06}, \cite{S15} 
for a computational detail. 
For a given data set $\{(x_i,y_i, z_i), i=1, \cdots, N\}$ with 
$(x_i,y_i)\in [0, 1]^2$ and $z_i= f(x_i,y_i)+ \epsilon_i, i=1, \cdots, N$  
with noises $\epsilon_i$ which may not be 
very small, the penalized least squares method (cf. \cite{L08} and \cite{LS09}) is to find 
\begin{equation}
\label{PLS}
\min_{s\in S^1_5(\triangle)} \sum_{i=1, \cdots, N} |s(x_i,y_i)- z_i|^2  
+ \lambda {\cal E}_2(s) 
\end{equation}
with $\lambda\approx 1$, where ${\cal E}_2(s)$ is the thin-plate energy functional 
defined as follows.
\begin{equation}
\label{E2}
{\cal E}_2(s) =\int_\Omega |\frac{\partial^2}{\partial x^2} s|^2 + 2
|\frac{\partial^2}{\partial x \partial y} s|^2 + |\frac{\partial^2}{\partial y^2} s|^2. 
\end{equation}  

Bivariate splines have been studied for several decades and they have been used for data fitting  (cf. \cite{L08}, \cite{LS09}, and 
\cite{LW13}, \cite{S15}, and \cite{WWL20}), 
numerical solution of partial differential equations (see, e.g. \cite{LL22}), \cite{S15}, 
and data denoising (see, e.g. \cite{LW13}).  In our computation, the triangulation $\triangle$ is 
the one obtained from uniformly refined the initial triangulation $\Delta_0$ three times, where
$\Delta_0$ is  obtained by dividing $[0, 1]^2$ into two triangles using its diagonal line.  Let us write 
$S^1_5(\triangle)=\hbox{span}\{\phi_1, \cdots, \phi_M\}$. For each $k=1, \cdots, dn$, we write 
$LKB_k= \sum_{j=1}^M c_{k,j}\phi_j$ with coefficients $c_{k,j}$'s being the solution of the linear system:
\begin{equation}
\label{keymatrixequation}
([ \sum_{i=1}^N \phi_\ell (x_i,y_i)\phi_j(x_i,y_k)]_{\ell,j=1,\cdots, M}+ \lambda E(\phi_\ell,\phi_j) ) [c_{k.j}] = [ \sum_{i=1}^N KB_k(x_i,y_i)\phi_\ell (x_i,y_i)]
\end{equation}
for $k=1, \cdots, nd$, where $E(\phi_\ell, \phi_j)$ is the matrix associated with energy functional ${\cal E}_2$.  For convenience, we 
call the matrix on the left-hand side by $A$ with $\lambda = 1$ fixed.

We now explain that the penalized least squares method can produce a good smooth approximation
of the given data.  
For convenience, let $S_{f,\epsilon}$ be the minimizer of (\ref{PLS})  and write  
$\|f\|_{\cal P}= \sqrt{\frac{1}{N}\sum_{i=1}^N |f(x_i,y_i)|^2}$ is the rooted mean squares (RMS) which is a semi-norm which is used to measure
the computational error.  
If $f\in C^2([0,1]^2)$, we have the following 
\begin{theorem}
	\label{main1}
	Suppose that $f$ is twice differentiable over $[0, 1]^2$.  Let $S_{f,\epsilon}$ be the minimizer of 
	(\ref{PLS}).  Then we have
	\begin{equation}
	\label{mainresult1}\|f- S_{f,\epsilon}\|_{\cal P}\le C\|f\|_{2,\infty}|\triangle|^2 + 2
	\|\epsilon\|_{\cal P} + \sqrt{\frac{\lambda}{N}} \sqrt{{\cal E}_2(f)}
	\end{equation}
	for a positive constant $C$ independent of $f$, degree $d$, and triangulation $\triangle$.  
\end{theorem} 

To prove the above result, let us recall the following minimal energy spline 
$S_f\in S^1_5(\triangle)$ of data
function $f$: letting $\triangle$ be a triangulation of $[0, 1]^2$ with vertices $(x_i,y_i),
i=1, \cdots, N$, $S_f$ is the solution of the following minimization:
\begin{equation}
\label{minenergy}
\min_{S_f\in S^1_5(\triangle)} {\cal E}_2(S_f): \quad S_f(x_i,y_i)= f(x_i,y_i), i=1, 
\cdots, N.
\end{equation}
Then it is known that $S_f$ approximates $f$ very well if $f\in C^2([0, 1]^2)$.  We have
\begin{theorem}[von Golitschek, Lai and Schumaker, 2002(\cite{VLS02})] 
\label{VLS}
	Suppose that $f\in C^2([0, 1]^2)$. Then for the minimal energy spline $S_f^*$ which is the
	solution of (\ref{minenergy}), 
	\begin{equation}
	\label{LSV}
	\| S_f^*- f\|_{\infty} \le C \|f\|_{2, \infty} |\triangle|^2
	\end{equation}
	for a positive constant $C$ independent of $f$ and $\triangle$, 
	where $\|f\|_{2,\infty}$ denotes the maximum norm of the second order derivatives of $f$ 
	over $[0, 1]^2$ and $\|S_f^*- f\|_{\infty}$ is the maximum norm of $S^*_f-f$ over $[0, 1]^2$.   
\end{theorem}

\begin{proof}[ of Theorem~\ref{main1}]
Recall that $S_{f,\epsilon}$ is the minimizer of (\ref{PLS}). 	We now use $S_f$ to have 
\begin{eqnarray*}
\|f- S_{f,\epsilon}\|_{\cal P} &\le & \|z- S_{f,\epsilon}\|_{\cal P} + \|\epsilon\|_{\cal P} 
\le   \sqrt{\|z- S_{f,\epsilon}\|^2_{\cal P} +\frac{\lambda}{N} {\cal E}_2(S_{f,\epsilon})} + 
\|\epsilon\|_{\cal P}\cr   
&\le & \frac{1}{\sqrt{N}} \sqrt{  \sum_{i=1}^N(z_i - S^*_f(x_i,y_i))^2 + \lambda {\cal E}_2(S^*_f)} + \|\epsilon\|_{\cal P} \cr 
&\le &  \|f- S_f^*\|_{\cal P} + \|\epsilon\|_{\cal P} + \sqrt{\frac{\lambda}{N} {\cal E}_2(f)} + \|\epsilon\|_{\cal P}\cr 
&\le&  C\|f\|_{2,\infty}|\triangle|^2 + 2\|\epsilon\|_{\cal P} +  \sqrt{\frac{\lambda}{N}} \sqrt{{\cal E}_2(f)}, 
\end{eqnarray*}
where we have used the fact $S_{f,\epsilon}$ is the minimizer of (\ref{PLS}) and the fact that ${\cal E}_2(S_f^*)\le {\cal E}_2(f)$ 
which can be found in \cite{VLS02} as well as the estimate in (\ref{LSV}).    
These complete the  theorem of this section. 
\end{proof}

Note that the constant $C$ is dependent on the smallest angle of the triangulation $\triangle$. As the domain of interest is $[0, 1]^2$ we use a uniform refinement of the triangulation of two 
triangles. For $d\ge 3$, we shall use a Delaunay  triangulation. So the constant $C$ is not very large.   
If $f$ is $C^2$ smooth, then $S_{f,\epsilon}$ will be a good approximation of $f$ when 
the size $|\triangle|$ of triangulation is small, the thin plate energy ${\cal E}_2(f)$ with $\lambda>0$ is 
bounded,  and the noises $\|\epsilon\|_{\cal P}$ is 
small even though a few individual noises $\epsilon_i$ can be large. Note also that $\epsilon$ can be made small 
by increasing the accuracy of the implementation of $\phi_q$. We also note that the proof of Theorem~\ref{VLS} can be 
straightforwardly extended to the multi-dimensional setting as soon as the degree $D$ of spline space is large enough.  

Now let us illustrate some examples of KB-splines and LKB-splines in Figure~\ref{f1}. One can see that the KB-splines are continuous but not smooth functions at all,  while the LKB-splines are very
smooth. With these LKB-splines in hand, we can approximate high dimensional continuous functions accurately. Let us report our numerical results in the next section.

\begin{figure}[h]
    \centering
    \begin{tabular}{cc} 
        \includegraphics[width=3in]{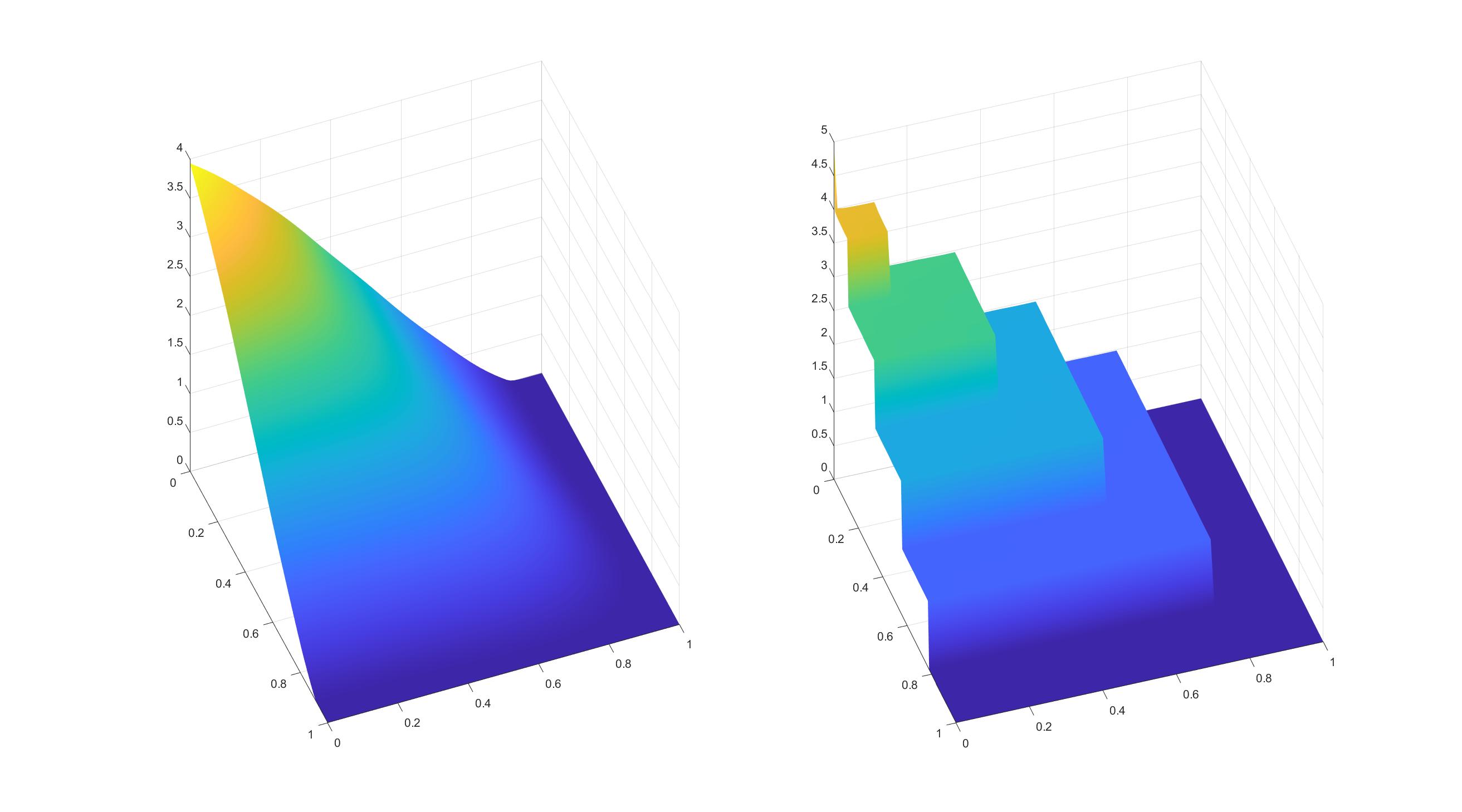}  &  \includegraphics[width=3in]{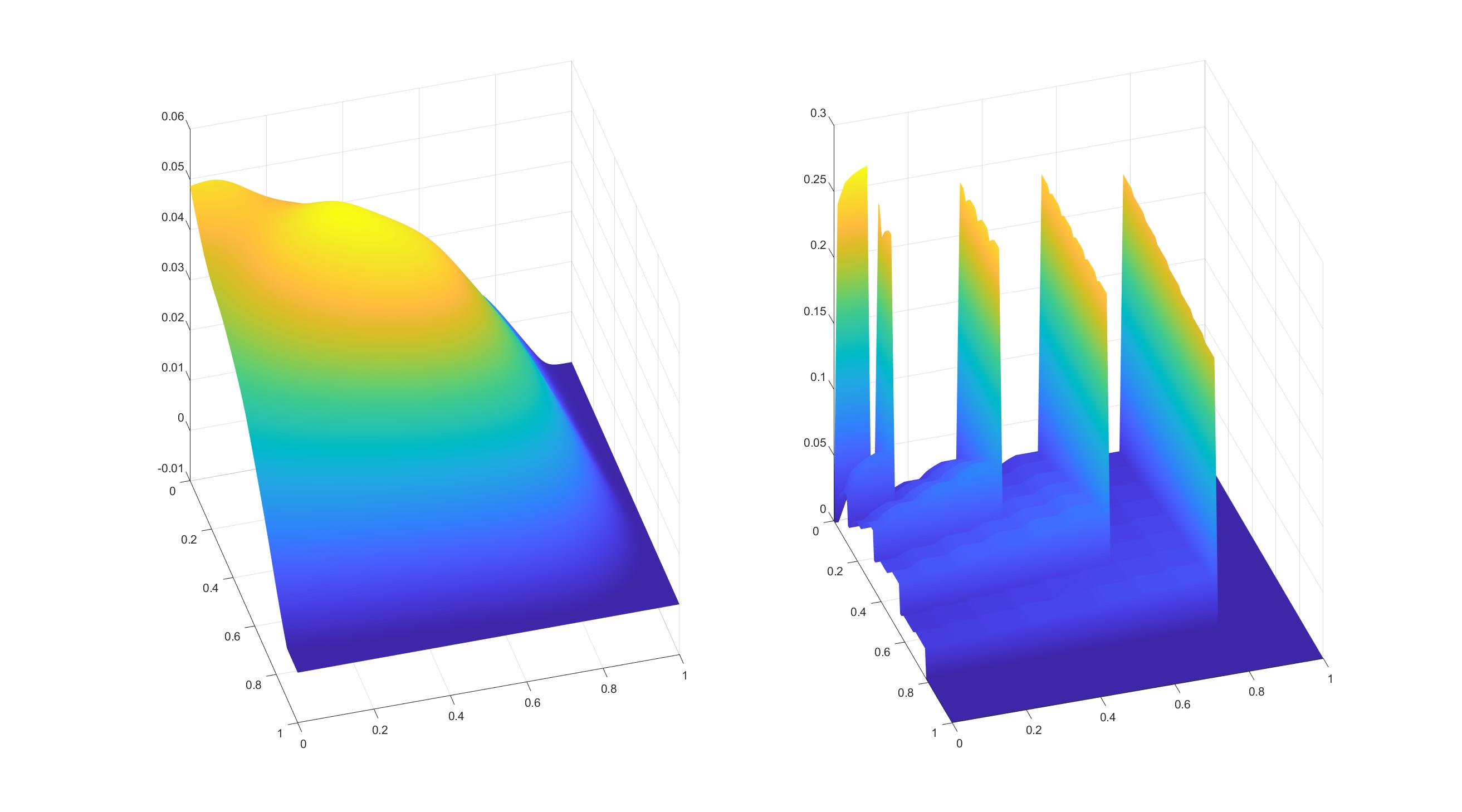} \\
	\includegraphics[width=3in]     {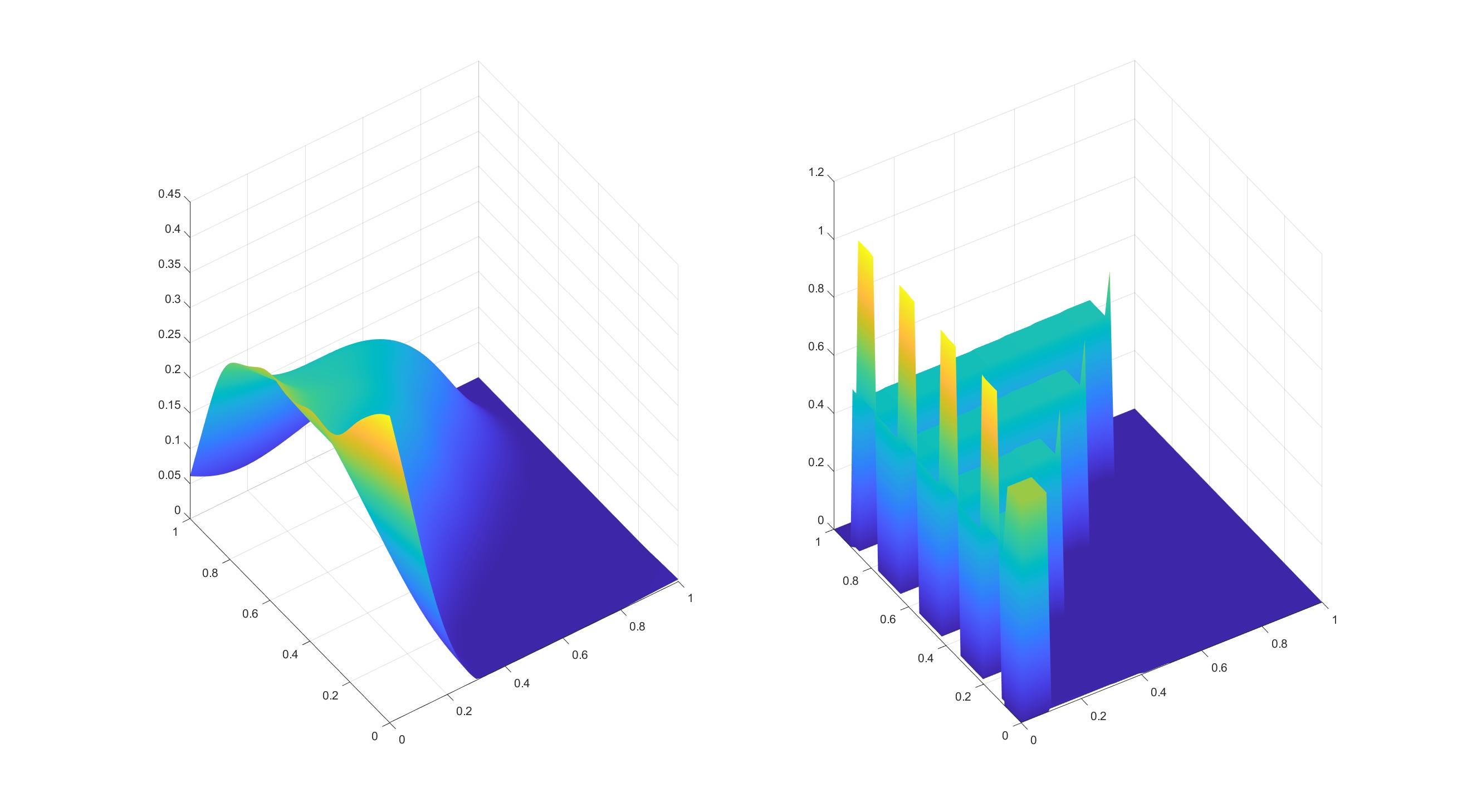} &
        \includegraphics[width=3in]{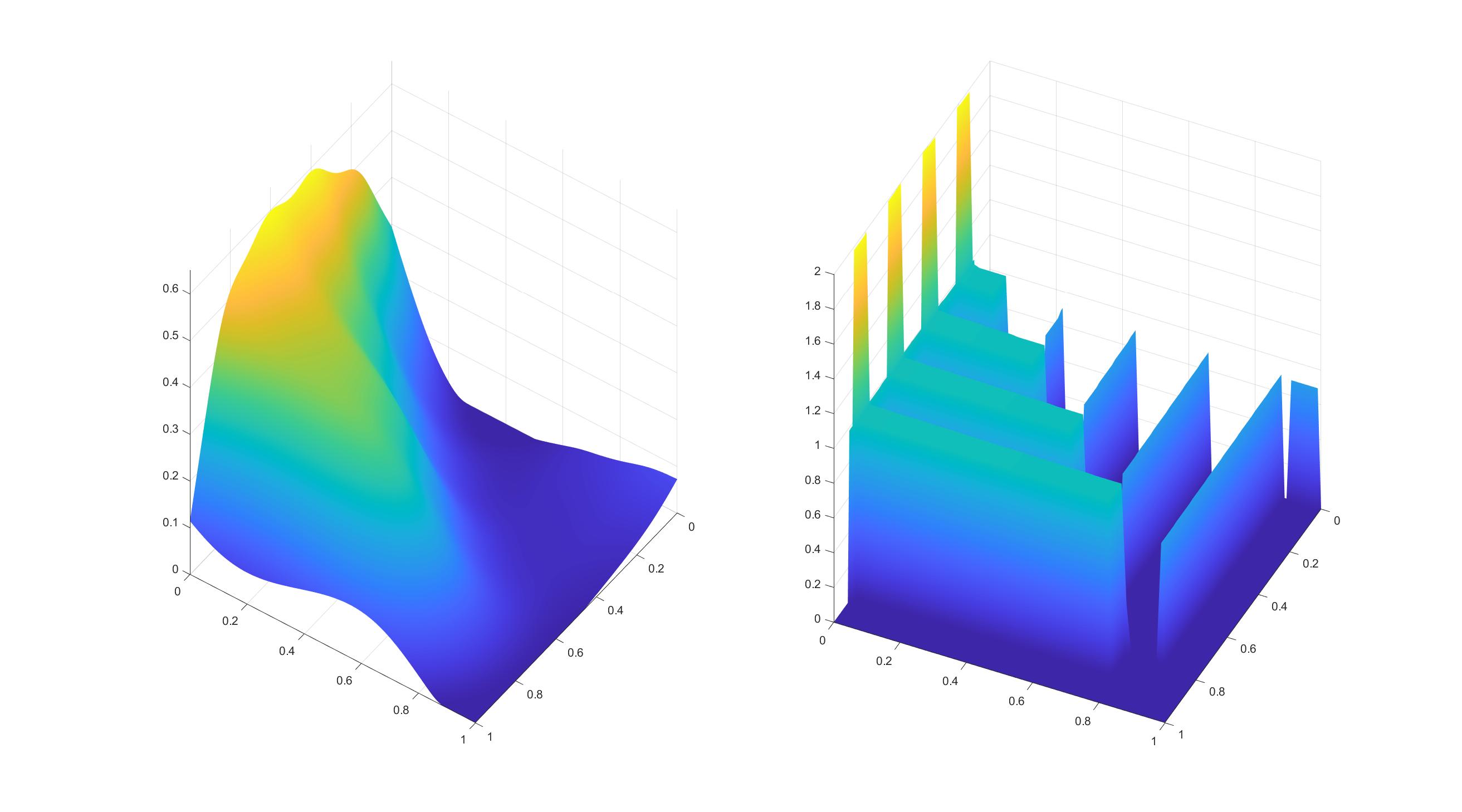} \\
    \end{tabular}
    \caption{Examples of LKB-splines (first and third columns) which are the smoothed version of the corresponding KB-splines (second and fourth columns). \label{f1}}
\end{figure}

% \begin{figure}[htbp]
% \includegraphics[width=3.25in]{LKB21a.jpg}
% \includegraphics[width=3.25in]{LKB21b.jpg}
% \caption{The LKB splines (the first and third) which are the smoothed version of KB splines (the second
% and the fourth) \label{f1}}
% \end{figure}  

% \begin{figure}[htbp]
% \includegraphics[width=3.25in]{LKB21c.jpg}
% \includegraphics[width=3.25in]{LKB21d.jpg}
% \caption{The LKB splines (the first and third) which are the smoothed version of KB splines (the 
% second and the fourth) \label{f2}}
% \end{figure} 

% \begin{figure}[htbp]
% \includegraphics[width=3.25in]{LKB21e.jpg}
% \includegraphics[width=3.25in]{LKB21f.jpg}
% \caption{The LKB splines (the first and third) which are the smoothed version of KB splines (the second
% and the fourth) \label{f3}}
% \end{figure} 

\section{Numerical Approximation by LKB-splines}
\label{sec4}
In this section, we will first show that the LKB-splines are dense in $C([0,1]^d)$. Then we demonstrate numerically that LKB-splines can approximate general continuous functions well based on $O(n^d)$ equally-spaced sampled data locations. Further, we use the matrix cross approximation technique to show that there are at most $O(nd)$ locations among those $O(n^d)$ locations are pivotal. Therefore, we only need the function values at those pivotal locations in order to achieve a reasonable  good approximation.

\subsection{The LKB-splines are Dense in $C([0, 1]^d)$}
We shall use discrete least squares method to approximate any continuous function $f$ over $[0, 1]^d$. Let ${\bf x}_i, i=1, \cdots, N$ be a set of discrete points over $[0, 1]^d$. 
For example  we may use $N=101^d$  equally-spaced points over $[0, 1]^d$. 
For any continuous function $f\in C([0,1]^d)$,  we use the function values at these data locations to find an 
approximation $F_n= \sum_{j=1}^{dn} c_j^* LKB_{n,j}$ by the discrete least squares method 
which is the solution of the following minimization
\begin{equation}
\label{DLS}
\min_{ c_j} \| f- \sum_{j=1}^{dn} c_j LKB_{n,j}\|_{\cal P},
\end{equation}
where $\|f\|_{\cal P}$ is the RMS semi-norm based on the function values $f$ over 
these $N=101^d$ sampled data points in $[0, 1]^d$. We shall report the accuracy $\|f- F_n(f)\|_{\cal PP}$, 
where $\|f\|_{\cal PP}$ is the RMS semi-norm based on $401^d$ function values.  
%The current computational power enables to do the numerical experiments for $d=2$ with $n=10, \cdots, 10,000$ and for $d=3$ with $n=10, \cdots, 1000$.   
Recall from Theorem~\ref{dense} and let $c_j(f)$ be the coefficients of the KB-spline approximation of $f$.   It is easy to see that 
\begin{equation}
\label{LKBapp}
\|f- F_n\|_{\cal P} \le  \|f  - \sum_{j=1}^{dn} c_k(f) LKB_{n,k}\|_{\cal P}.
\end{equation}
Writing $LKB_{n,k} = \sum_{j=1}^M c_{k,j} \phi_j$ as in the previous section with the coefficient vector 
\begin{equation*}
[c_{k,j}]= A^{-1}[ \sum_{i=1}^N KB_k(x_i,y_i)\phi_j(x_i,y_i)], 
\end{equation*}
where $A$ is the matrix as in (\ref{keymatrixequation}),  we see that 
\begin{eqnarray*}
 \sum_{k=1}^{dn} c_k(f) LKB_{n,k} &=&\sum_{k=1}^{dn} c_k(f) \sum_{j=1}^M c_{k,j} \phi_j\cr
 &=& \sum_{k=1}^{dn} c_k(f) \sum_{j =1}^M A^{-1}[ \sum_{i=1}^N KB_{n,k}(x_i,y_i)\phi_j(x_i,y_i)] \phi_j \cr
 &=&  \sum_{j =1}^M A^{-1}  [\sum_{i=1}^{N}  \sum_{k=1}^{dn} c_k(f) KB_{n,k}(x_i,y_i)\phi_j(x_i,y_i)] \phi_j \cr
 &=&  \sum_{j =1}^M A^{-1}  [\sum_{i=1}^{N}( f(x_i,y_i)+O(\epsilon))\phi_j(x_i,y_i)] \phi_j, 
\end{eqnarray*}
where we have used the proof of Theorem~\ref{dense}. i.e.  $\sum_{k=1}^{dn} c_k(f) KB_{n,k}(x_i,y_i) = f(x_i,y_i)+O(\epsilon)$.  Now we note that the 
right-hand side of the equations above is simply $S_{f,\epsilon}$. That is, 
\begin{equation}
 \sum_{j=1}^{dn} c_k(f) LKB_{n,k} = S_{f, \epsilon}.
\end{equation}
By Theorem~\ref{main1}, we conclude the following 
\begin{theorem}
\label{newmain}
Suppose that $f$ is twice differentiable over $[0, 1]^2$.  Let $F_n$ be the discrete least squares approximation of $f$ defined in (\ref{DLS}). Suppose that the points $\bfx_i=(x_i,y_i)$ for 
(\ref{DLS}) are the same as the points for denoising KB-splines to have the LKB functions.  
Then 
\begin{equation}
\label{newestimate}
\|f- F_n\|_{\cal P} \le  C\|f\|_{2,\infty} |\triangle|^2+ 2 \|\epsilon\|_{\cal P} + \frac{1}{\sqrt{N}} \sqrt{{\cal E}_2(f)}
\end{equation}
for a  positive constant $C$ independent of $f$ and triangulation $\triangle$.  
\end{theorem} 

Although our discussion above is based on the case $d=2$, all the proofs can be extended to the dimensionality $d>2$.   We next explain that the computation of discrete least squares
(\ref{DLS}) can be done based on much simpler data points and function values in the following subsection.

\subsection{The pivotal data locations for breaking the curse of dimensionality}

\begin{figure}
\centering
\includegraphics[width=0.4\textwidth]{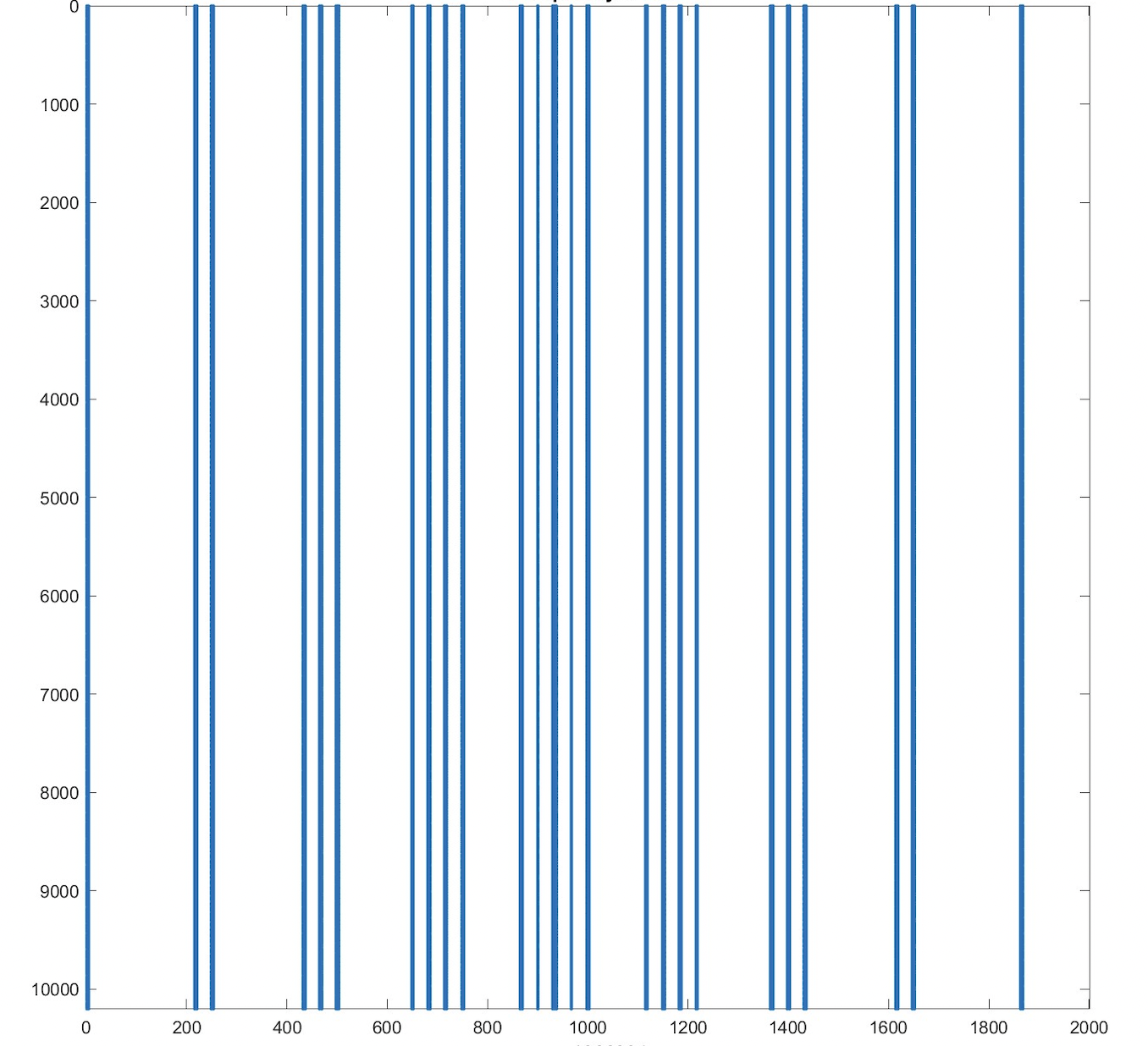} 
\vspace{-3mm}
\caption{{The sparsity pattern of data matrix for $n=1000$.}  \label{spymatrix}}
\end{figure}

For convenience, let us use $M$ to indicate the data matrix associated with the discrete least squares problem (\ref{DLS}). In other words, for $1\leq j\leq dn$, the $j$th column $M(:,j)$ consists of $ \{LKB_{n,j}(\mathbf{x}_i)\}$ where $\mathbf{x}_i\in [0,1]^d$ are those $41^d$ equally-spaced sampled points in 2D or 3D. 
Clearly, the experiment above requires $41^d$ data values which suffers from the curse of dimensionality. However, we in fact do not need such many data values.  The main reason is that the data matrix
$M$ has many zero columns or near zero columns due to the fact that for many $i=1, \cdots, nd$, the locations from $41^d$
equally-spaced points do not fall into the support of linear B-splines $b_{n,i}(t)$, $t\in [0,d]$, based on the map $z_q$.  The structures of $M$ are shown in Figure~\ref{spymatrix} for the case of $n=1000$ when $d=2$. 

% \begin{figure}[h]
%     \centering
%     \includegraphics[width=0.35\textwidth]{spymatrix2d2000.png} 
% \vspace{-2mm}
% \caption{{The sparsity pattern of data matrix for $n=1000$.}  \label{spymatrix}}
% \end{figure}

% \begin{wrapfigure}{r}{0.42\textwidth}
%     \centering
%     \includegraphics[width=0.42\textwidth]{SparsePoly.png}
%     \caption{The surfaces and errors of QOMP in reconstruction of biavariate sparse polynomials (Using $144$ basis and $32$ sampling points, the sparsity is $5$).
%     }
%     \label{SparsePoly}

% %\vspace{5pt}
% \end{wrapfigure}

That is, there are many columns in $M$ whose entries  are
zero or near zero. Therefore, there exists a sparse solution to the discrete least squares fitting problem. We adopt the well-known orthogonal matching pursuit (OMP) (cf. e.g. \cite{LW21}) to find a solution. For convenience, let us
explain the sparse solution technique as follows. Over those $41^d$ points $\bfx_i\in 
[0, 1]^d$, the columns in the matrix

\begin{equation}
\label{thematrix}
M=[ LKB_{n,j}(\bfx_i) ]_{i=1,\cdots, 41^d, j=1,\cdots, dn}
\end{equation}
are not linearly independent.  Let $\Phi$ be the normalized matrix of $M$ in (\ref{thematrix}) and 
$\bfb=[f(\bfx_i)], i=1, \cdots, 41^d$. Write ${\bf c}=(c_1, \cdots, c_{dn})^\top$, we look for
\begin{equation}
\label{CompressiveSensing}
 \min \|{\bf c}\|_0:  \Phi {\bf c}= {\bf b}
\end{equation}
where $\|{\bf c}\|_0$ stands for the number of nonzero entries of ${\bf c}$.   See many
numerical methods in (\cite{LW21}). The near zero columns in $\Phi$ also tell us that the data matrix associated with (\ref{DLS}) of size $41^d\times dn$
is not full rank $r< dn$. The LKB-splines associated with these near zero columns do not play a role. Therefore, we do not need all $dn$ 
LKB-splines. 
% This fact is also the reason that one does not need all $101^d$ data locations and function values due to the linear independence of the associated KB splines. 
Furthermore, let us continue to explain that many data locations among these $41^d$ locations do not play an essential role. 
%Certainly, when we use randomly sampled points in $[0, 1]^2$, some of these LKB splines may play a role and 
%hence, the randomly sampled points may be better for doing approximation than using the equally-spaced points.   

To this end, we use the so-called matrix cross approximation (see  \cite{GT01}, \cite{GOSTZ08}, \cite{MO18}, 
\cite{GH17}, \cite{KS16}, \cite{ALS23} and the literature therein). Let $r\ge 1$ be a 
rank of the approximation. It is known (cf. \cite{GT01}) that when $M_{I,J}$ of size $r\times r$ has the maximal volume among all submatrices of $M$ of size $r\times r$, we have 
\begin{equation}
\label{GT01}
\| M- M_{:,J} M_{I,J}^{-1} M_{I,:}\|_C \le (1+r)\sigma_{r+1}(M),
\end{equation}
where $\|\cdot \|_C$ is the Chebyshev norm of matrix and $\sigma_{r+1}(M)$ is the 
$r+1$ singular value of $M$, $M_{I,:}$ is the row block of $M$ associated with the
indices in $I$ and $M_{:,J}$ is the column block of $M$ associated with the indices in $J$. The volume of a square
matrix $A$ is the absolute value of the determinant of $A$.  

Note that the estimate in (\ref{GT01}) is not very good as when $\sigma_k(M) =O(1/k)$, there is no approximation at all. The estimate is recently improved in \cite{ALS23} which is given in (\ref{lsqInequ1}).

We mainly  find a submatrix $M_{I,J}$ of $M$ such that $M_{I,J}$ has the maximal volume among all $r\times r$ 
submatrices of $M$. In practice, we use the concept called dominant matrix to replace the maximal 
volume and then there are several aglorithms, e.g. maxvol algorithm, available in the literature.  We use   a few greedy based maximal volume search algorithms
developed in  \cite{ALS23}. These greedy based maxvol algorithms enable us to find a good submatrix $M_{I,J}$ 
which leads to solve a much simpler discrete least squares problem
\begin{equation}
\label{DLS3}
[M_{I.J}. M_{I,J^c}] \widehat{\bfx} =  \bff_I 
\end{equation}
where  $\bff=[\bff_I;\bff_{I^c}]$ and 
\begin{equation}
    M = 
    \begin{bmatrix}
    M_{I,J} & M_{I,J^c} \\
    M_{I^c, J} & M_{I^c, J^c}
\end{bmatrix}.
\end{equation}
according to \cite{ALS23} or simply 
\begin{equation} 
\label{DLS2}
M_{I,J}\widehat{\bfx} =  \bff_I, 
\end{equation}
 as $M_{:,J^c}\approx 0$  according to our $X_{data}$ above. We use the same 
 analysis in  \cite{ALS23} to have
 \begin{equation} \label{lsqInequ1}
    \|M_{I^c,J^c}-M_{I^c,J}M_{I,J}^{-1}M_{I,J^c}\|_C\leq \frac{(r+1)\sigma_{r+1}(M)}{\sqrt{1+\sum_{k=1}^{r}\frac{\sigma_{r+1}^2(M)}{\sigma_k^2(M)}}}.
\end{equation} 
In fact, we choose the rank $r$ of $M$ so that the above term is zero.  Furthermore, 
letting $\delta = M{\bf x}- \bff $ be the residual vector of the discrete least squares
approximation (\ref{DLS}), we write $\delta= [\delta_I; \delta_{I^c}]$. Then 
the detail calculation in \cite{ALS23} shows that 
\begin{align*}
    M_{I^c,J}\mathbf{\hat{x}}-\mathbf{f}_{I^c}&=M_{I^c,J}M_{I,J}^{-1}\mathbf{f}_I-\mathbf{f}_{I^c} \\
    & = M_{I^c,J}M_{I,J}^{-1}\delta_I+
    \begin{bmatrix}
        0_{I^c,J} & M_{I^c,J}M^{-1}_{I,J} M_{I,J^c}- M_{I^c,J^c}
    \end{bmatrix} \mathbf{x}_b+\delta_{I^c}\\
    &= M_{I^c,J}M_{I,J}^{-1}\delta_I+\delta_{I^c}
\end{align*}
since the middle term on the right-hand side is zero as explained above.  For simplicity, 
the solution  $\widehat{\bfx}$ of (\ref{DLS2}) with size $r \times 1$ is also viewed as a
vector in the original size $nd \times 1$ with zeros over the index set $I^c$.  Then the root 
mean square error 
\begin{eqnarray*}
\|M \widehat{\bfx}- \bff\|_{\cal P} = \| M_{I^c,J} \widehat{\bfx}- \bff_{I^c}\|_{\cal P} 
\le  \|M_{I^c,J}M_{I,J}^{-1}\delta_I \|_{\cal P}+ \|\delta_{I^c}\|_{\cal P}.
\end{eqnarray*}  
The first term on the right-hand side can be estimated as follows. 
By the property of $M_{I,J}$, we know that all the entries of $M_{I^c,J}M_{I,J}^{-1}$ 
are less than or equal to $1$. So the $\ell_2$ norm 
$ \|M_{I^c,J}M_{I,J}^{-1}\delta_I\|_2 \le \sqrt{r \|\delta_I\|^2_2}$ and hence, 
$$
 \|M_{I^c,J}M_{I,J}^{-1}\delta_I \|_{\cal P} \le \sqrt{r/N}\|\delta_I\|_2. 
$$
Hence, the solution $\widehat{\bfx}$ in (\ref{DLS2}) is a good approximation of $\bfx_f$, the least squares solution vector (\ref{DLS}).   More precisely, we have  

\begin{theorem} [\cite{ALS23}]
	\label{mjlaiNov182022} 
Let the residual vector $ \delta=M\bfx_f- \bff$ and write 
$\delta=[\delta_I;\delta_{I^c}]$.  Then we have 
	\begin{equation}
	\label{estimate8a}
	\|M\widehat{\bfx}- \bff\|_{\cal P} \le \sqrt{r/N}\|\delta_I\|_2+ \|\delta_{I^c}\|_{\cal P}.
	\end{equation}
\end{theorem} 

Let $\widehat{F}_n$ be the associated LKB spline approximation of $f$, i.e. $\widehat{F}_n=  [LKB_{n,1}, \cdots, LKB_{n,nd}]\widehat{\bfx}$. We now show that 
$\widehat{F}_n \approx f$.   
First, recall Theorem~\ref{newmain},  for a continuous
function $f\in C^2([0,1]^2)$, the solution $F_n$ of the discrete least squares approximation (\ref{DLS}) approximates $f$ very well. That is, 
the RMS error $\|f- F_n\|_{\cal P}=\|\delta\|_{\cal P}$ with
\begin{equation}
\label{estimate8}
\|\delta\|_{\cal P} \le C \|f\|_{2,\infty} |\triangle|^2+ 2\|\epsilon\|_{\cal P} + 
\sqrt{\frac{1}{N}} \sqrt{{\cal E}_2(f)}
\end{equation}
from Theorem~\ref{dense}.  Thus, 
\begin{eqnarray}
\label{newestimate3}
\| \widehat{F}_n - f\|_{\cal P} \le \| M\widehat{\bfx}- \bff\|_{\cal P} \le \sqrt{r}\|\delta\|_{\cal P}
  + \|\delta\|_{\cal P} . 
\end{eqnarray}
When $f$ is Kolmogorov-Lipschitz continuous with Lipschitz constant $L$, we know $\epsilon=O(L/n)$ from Corollary~\ref{newmain3}. 
Combining the estimate in (\ref{estimate8}), we have
\begin{equation}
\label{estimate88}
\| \widehat{F}_n - f\|_{\cal P} \le  (\sqrt{r}+1)C\|f\|_{2,\infty}|\triangle|^2 
+ O((\sqrt{r}+1)L/n)  + \sqrt{\frac{2r}{N}} \sqrt{{\cal E}_2(f)},  
\end{equation}
where $r$ is the rank of $X_{data}$ and is strictly less than $nd$. Assume that $\|f\|_{2,\infty}$ is bounded. 
We can choose $|\triangle|$ so small so that 
the first term on the right of (\ref{estimate88}) is small, that is, make 
$\sqrt{nd}|\triangle|^2$ small enough. The second term $\sqrt{r}/n L\le \sqrt{2/n}L$ can be small if $n$ is 
large enough. The third term $\sqrt{2r/N} \le 2\sqrt{nd/N}$ can be small if ${\cal E}_3(f)$ is bounded as 
$N\gg n$ and $N\gg d$.  

In addition to the equally-spaced points from $[0, 1]^d$, we can also use a set of randomized point locations over $[0, 1]^d$. The above discussion can also be applied. 
 Let us conclude all the discussion above and write down one of our main results in this paper. 
\begin{theorem}
\label{mjlai2024}
For any $n\gg 1$, there exists a set ${\cal D}_n$ of pivotal data locations over $[0, 1]^d$ with 
the cardinality $|{\cal D}_n|<nd$ such that  for any $f\in C([0, 1]^d)$, the  
discrete least squares problem (\ref{DLS2}) based on the function measurements 
$\bff_I=[f(\bfx_i), \bfx_i \in {\cal D}_n] $ is solved to obtain a LKB spline $\widehat{F}_n$. 
Then if $f\in C^2([0,1]^d)$,  the LKB spline $\widehat{F}_n$  
approximates $f$ very well in the sense of (\ref{newestimate3}). Furthermore, if $f$ is also 
Kolmogorov-Lipschitz continuous with     Lipschitz constant $L$, then the LKB spline 
$\widehat{F}_n$ approximates $f$ in the sense of (\ref{estimate88}).  
\end{theorem}

Let us call the data locations associated with row indices $I$ the pivotal data locations. Also, we will call
such data locations magic data locations. Two examples of pivotal data locations are shown in Figure~\ref{pivotal}. It is worthwhile to point out that such a set of pivotal locations is only dependent on the knot partition 
of $[0,d]$ when numerically building KB-splines, the sampled data when constructing LKB-splines, and the smoothing parameters for converting KB-splines to LKB-splines. However, such a set of pivotal locations is 
independent of any testing functions or the target function to approximate.  

\begin{figure}[htbp] 
    \centering
    \begin{tabular}{cc}
    \includegraphics[width=0.45\textwidth]{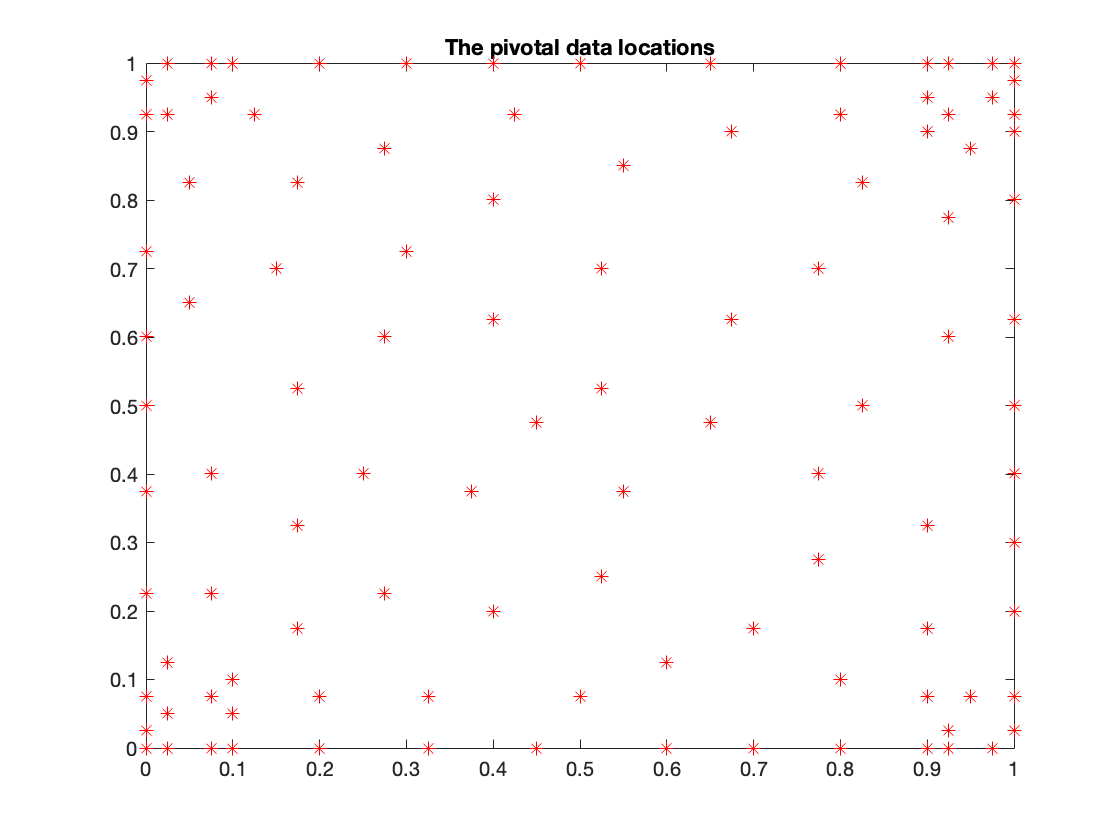} &
    \includegraphics[width=0.45\textwidth]{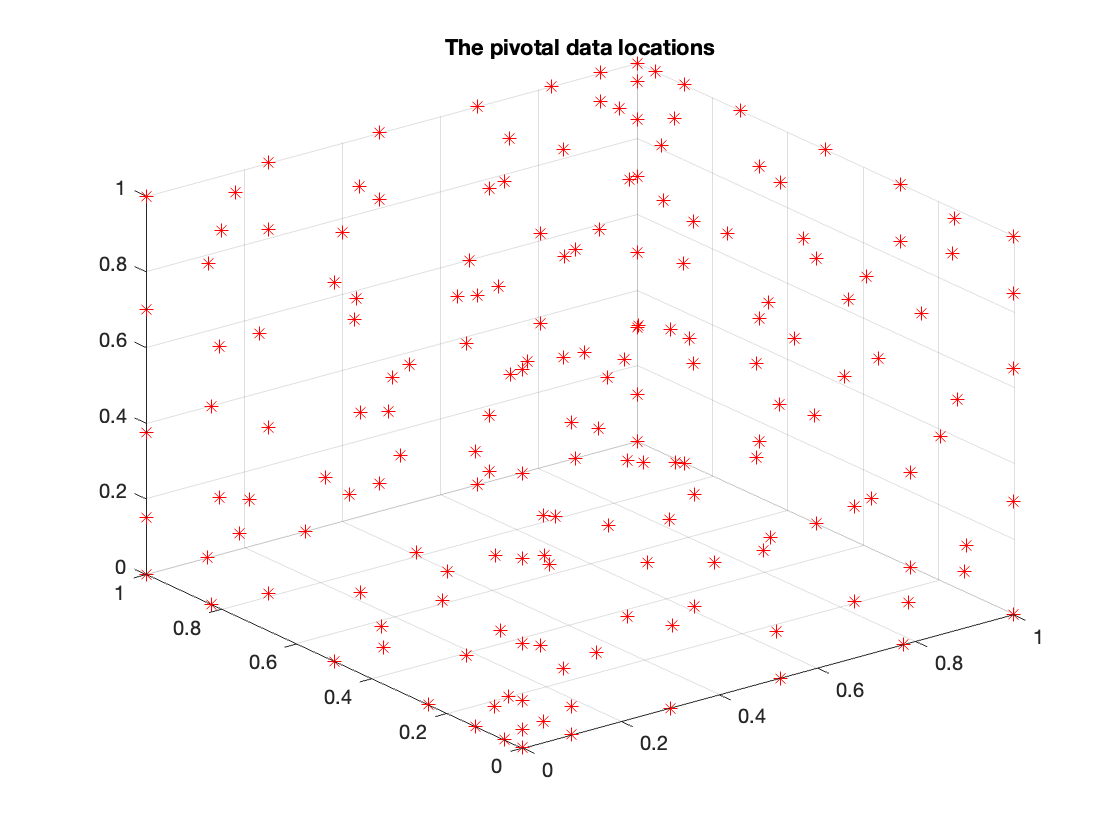} 
    \end{tabular}
    
    \caption{Pivotal data at 99 locations (selected from $41^2$ equally-spaced locations) in 2D and 178 locations (selected from $41^3$ equally-spaced locations) in 3D for $n=100$. \label{pivotal}}
\end{figure} 

The above discussion shows that we only need to $nd$ LKB-splines $LKB_{n,j}, j=1, 
\cdots, nd$ and the measurement of any function $f$ over a set of pivotal data locations with the 
number of the pivotal data location set less than $nd$ 
to obtain LKB-spline approximation  $\widehat{F}_n$ of $f$ over a dense set over $[0, 1]^d$
with a number $N\gg nd$.   These finish our approach to 
approximate smooth continuous functions in the multi-dimensional setting. This 
approach does not suffer from the curse of dimensionality although the computation of the
denoising process to obtain LKB-splines will take a lot of computer power when $d\gg 2$.  Nevertheless, 
$LKB_{n,j}, j=1, \cdots, nd$ can be computed independently and the measurements of the
values from $KB_{n,j}, j=1, \cdots, nd$ are obtained based on their computation. Essentially, we need some big companies like Google and
Microsoft to help us obtain these LKB-splines $B_{nj}$ when $d\gg 2$. The rest of us can simply use them to 
approximate any  Kolgoromov-Lipschitz continuous functions  without the curse of dimensionality.  

\subsection{Numerical Approximation by LKB-splines from a Dimension Reduction Point of View}
The approximation by LKB-splines can be thought as a dimension reduction problem in the following sense:  Given $n$ sets of data which is in a high-dimensional space $\mathbb{R}^N$ 
with $N\gg n$, find if the given $n$ sets of data lie on a lower dimensional manifold, say dimensionality $k$ with $k\ll n$. 
For example, consider a real-life situation that one uses a camera
to take a picture over a rectangular domain of interest.  Each picture is an image of size $512\times 512\times 3$. For convenience, let us focus on pictures
which are  black-and-white images. We can recast the dimension reduction problem in the following format: 
We are given many images of $512\times 512$ and would like to see if they can be approximated by a lower dimensional
subspace, e.g., the spanned by some LKB functions within a given tolerance $\epsilon$. Some literature related to this direction can be found in \cite{chen2022nonparametric, chui2018deep, Cloninger2021, Nakada2020, Schmidt2019, Shen2025}

Consider the discrete least squares approximation (DLS) which was explained in a previous section. That is, we build up a DLS matrix based on 
LKB functions. 
\begin{equation}
\label{imageapp}
M=  [{\rm LKB}_{n,j}(\mathbf{x}_i)]_{1\le i \le 512^2, j=1, \cdots, 2n}.
\end{equation}
As discussed before, the matrix $M$ is not of full column rank and we can find a pivotal data set $\mathbf{x}_i, i\in I$, where $I\subset [1, 2, \cdots, N]$
with $N=512^2$ and $\#(I)\leq 2n$. Let $\epsilon$ be the error estimated on the right-hand side of (\ref{estimate88}), we know that the picture functions 
are in a low-dimensional subspace if the triangulation size $|\triangle|$ is small enough 
vs the maximum norm of these picture functions, the Kolmogorov-Lipschitz constant $L$ is small enough vs $n$, and the energy of the picture functions
is small vs $N=512^2/n$.  It is easy to understand that the size $|\triangle|$ can be made small as long as we have a good computer with a large CPU. 
Also, we know that the energy of the function $f$ is related to the $L^2$ norm of the second-order derivatives of $f$.  The only problem left is how to know
the Kolmogorov-Lipschitz constant $L$. 

In the following,  we design  a numerical experiment to determine the Kolmogorov-Lipschitz constant $L$ for various functions. Mainly, we use 
Theorem~\ref{mjlai2019} to check the rate of convergence. Similarly, we can use Theorem~\ref{mjlai2022} to 
see the Kolmogorov modulus of smoothness.  
%demonstrate the power of LKB-splines for approximating functions in $C([0,1]^d)$ for $d\ge 2$. 
In this way,  we can determine if these picture functions are in a low-dimensional subpace spanned by some LKB functions.   
%For the numerical experiments, we choose $41^d$ equally-spaced points 
%$\bfx_i \in [0, 1]^d$. For any continuous function $f\in C([0,1]^d)$, 
%we first use the function values at these data locations to experiment an 
%approximation $F_n= \sum_{j=1}^{dn} c_j LKB_{n,j}$ by using discrete least squares (DLS) method via (\ref{DLS}). 
%Then we use the magic data locations to interpolate the given function values at these data locations.  We will 
%show the convergence in the RMS semi-norm over $101^d$ equally-spaced points in $[0, 1]^d$.  
% In other words, 
% we solve the following minimization problem:
% \begin{equation}
% \label{DLS}
% \min_{ c_j} \| f- \sum_{j=1}^{dn} c_j LKB_{n,j}\|_{\cal P},
% \end{equation}
% where $\|f\|_{\cal P}$ is the RMS semi-norm based on the function values $f$ over 
% these $41^d$ sampled data points in $[0, 1]^d$. We shall report the accuracy $\|f- F_n(f)\|_{\cal PP}$, 
% where $\|f\|_{\cal PP}$ is the RMS semi-norm based on $101^d$ function values. 
First of all, we note that the current computational power that the authors possess enables them to do the numerical experiments for $d=2$ 
with $n=100, 200, \cdots, 10,000$ and for $d=3$ with $n=100, 200, \cdots, 1000$.   

For $d=2$, we use the following 10 testing functions across different families of continuous functions to report the computational accuracy. They
are among 100 testing functions we have experimented so far.  For convenience, we use $N=101^2$ instead of $512^2$. The rate of convergence will
be shown. 
\begin{eqnarray*} 
f_1&=& (1+2x+3y)/6; \qquad f_2= (x^2+y^2)/2; \qquad  f_3= xy; \cr 
f_{4}&=& (x^3+y^3)/2;   \qquad 
f_{5}= 1/(1+x^2+y^2) ;\cr 
f_{6}&=& \cos(1/(1+xy));  \qquad 
f_{7} = \sin(2\pi(x+y));\cr
f_{8}&=&\sin(\pi x)\sin(\pi y);  \qquad
f_{9} = \exp(-x^2-y^2);\cr
f_{10}&=& \max(x-0.5,0)\max(y-0.5,0);
% f_{19}&=& ((1+2x+4y)/7)^{10};\cr
% f_{20}&=& \sin(\pi(x^2+y^2));\cr
\end{eqnarray*}

% \begin{table}[thpb]
% 	\caption{RMSEs (computed based on $101^2$ equally-spaced locations) of the DLS fitting (\ref{DLS}) based on $41^2$ equally-spaced location and pivotal location in 2D. \label{2Dex1}}
% 	\centering 
% \vspace{2mm}
% 	\begin{tabular}{c|cc|cc|cc}
%  \toprule
% 	 & \multicolumn{2}{|c|}{$n=100$} & \multicolumn{2}{|c|}{$n=1000$} & \multicolumn{2}{c}{$n=10000$} \\
%         \cmidrule{2-7}
%       \# sampled data  & $41^2$ & 54 & $41^2$ & 105 & $41^2$ & 521 \\
%         \midrule
% 		$f_1$ & 1.67e-05 & 2.60e-05 & 5.79e-06 & 1.02e-05  & 5.14e-07 & 1.21e-06 \\
% 		$f_2$ & 4.19e-04 & 8.92e-04 &  1.17e-04 & 2.61e-04 & 2.83e-05 & 6.62e-05 \\
% 		$f_3$ & 1.09e-04 & 2.19e-04 & 3.57e-05 & 7.46e-05 & 2.20e-05 & 5.67e-05  \\  
% 		$f_4$ & 7.67e-04 & 1.70e-03  & 2.10e-04 & 5.11e-04 & 4.99e-05 & 1.11e-04  \\  
% 		$f_5$ & 2.28e-04 & 5.04e-04 & 6.69e-05 & 1.47e-04 & 1.93e-05 & 4.08e-05    \\    
% 		$f_6$ & 2.52e-04 & 6.51e-04 & 7.97e-05 & 1.94e-04 & 1.43e-05 & 2.73e-05   \\  
% 		$f_7$ & 7.05e-02 & 1.32e-01 & 7.80e-03 & 2.25e-02 & 1.30e-03 & 3.30e-03   \\
% 		$f_8$ & 1.50e-03 & 2.29e-03 & 3.73e-04 & 1.01e-03 & 1.69e-04 & 4.37e-04  \\ 
% 		$f_9$ & 3.49e-04 & 7.97e-04 & 8.25e-05 & 1.98e-04 & 2.48e-05 & 5.52e-05   \\  
% 		$f_{10}$ & 2.02e-03 & 3.79e-03 & 7.77e-04 & 1.82e-03 & 1.68e-04 & 4.10e-04 \\
%   \bottomrule   
% 	\end{tabular}
% \end{table}

\begin{table}[thpb]
	\caption{RMSEs (computed based on $401^2$ equally-spaced locations) of the DLS fitting (\ref{DLS}) based on $101^2$ equally-spaced location and pivotal location in 2D. \label{2Dex1}}
	\centering 
\vspace{2mm}
	\begin{tabular}{c|cc|cc|cc}
 \toprule
	 & \multicolumn{2}{|c|}{$n=100$} & \multicolumn{2}{|c|}{$n=1000$} & \multicolumn{2}{c}{$n=10000$} \\
        \cmidrule{2-7}
      \# sampled data  & $101^2$ & 99 & $101^2$ & 187 & $101^2$ & 879 \\
        \midrule
		$f_1$ & 1.53e-05 & 2.54e-05 & 7.55e-06 & 1.41e-05  & 5.32e-07 & 1.26e-06 \\
		$f_2$ & 1.41e-04 & 3.06e-04 &  6.48e-05 & 1.42e-04 & 1.72e-05 & 4.33e-05 \\
		$f_3$ & 8.31e-05 & 1.66e-04 & 3.45e-05 & 6.97e-05 & 2.16e-05 & 5.31e-05  \\  
		$f_4$ & 2.40e-04 & 4.71e-04  & 1.16e-04 & 2.52e-04 & 3.02e-05 & 7.81e-05  \\  
		$f_5$ & 8.99e-05 & 1.98e-04 & 4.59e-05 & 1.17e-04 & 1.01e-05 & 2.74e-05    \\    
		$f_6$ & 1.13e-04 & 2.62e-04 & 4.64e-05 & 9.04e-05 & 1.02e-05 & 2.13e-05   \\  
		$f_7$ & 1.29e-02 & 3.13e-02 & 3.77e-03 & 1.03e-02 & 7.13e-04 & 1.43e-03   \\
		$f_8$ & 7.38e-04 & 1.49e-03 & 2.85e-04 & 5.95e-04 & 8.75e-05 & 2.21e-04  \\ 
		$f_9$ & 1.16e-04 & 2.96e-04 & 5.93e-05 & 1.52e-04 & 1.73e-05 & 4.11e-05   \\  
		$f_{10}$ & 9.76e-04 & 1.87e-03 & 5.02e-04 & 8.40e-04 & 1.28e-04 & 2.42e-04 \\
  \bottomrule   
	\end{tabular}
\end{table}

For $d=3$, we choose the following 10 testing functions across different families of continuous functions to check the computational accuracy. 
The computational results for $d=2$ and $d=3$ are reported in Tables \ref{2Dex1} (those columns associated with $101^2$) and \ref{3Dex2} (those columns associated with $41^3$) respectively.
\begin{eqnarray*} 
f_1 &=& (1+2x+3y+4z)/10; \qquad f_2 = (x^2+y^2+z^2)/3; \qquad f_3 = (xy+yz+zx)/3; \cr 
f_4 &=& (x^3y^3+y^3z^3)/2; \qquad f_5 = (x+y+z)/(1+x^2+y^2+z^2); \cr 
f_6 &=& \cos(1/(1+xyz)); \qquad  f_7 = \sin(2\pi(x+y+z)); \cr 
f_8 &=& \sin(\pi x)\sin(\pi y)\sin(\pi z); \qquad  f_9 = \exp(-x^2-y^2-z^2);\cr 
f_{10} &=& \max(x-0.5,0)\max(y-0.5,0)\max(z-0.5,0);
\end{eqnarray*} 

Next, we present the numerical results in Table~\ref{2Dex1} and~\ref{3Dex2} to demonstrate that the numerical approximation results based on pivotal data locations (those columns associated with $99,187,879$ (in Table~\ref{2Dex1}) and $178,331,643$ (in Table~\ref{3Dex2}) have the same order as the results based on data locations sampled on the uniform grid with $101^2$ or $41^3$ locations. 
We therefore conclude that the curse of dimensionality for 2D and 3D function approximation 
is broken  if we use LKB-splines with pivotal data locations. % when $f$ is not very oscillated.

\begin{table}[thpb]
	\caption{RMSEs (computed based on $101^3$ equally-spaced locations) of the DLS fitting (\ref{DLS}) based on $41^3$ equally-spaced location and pivotal location in 3D. \label{3Dex2}}
	\centering 
\vspace{2mm}
	\begin{tabular}{c|cc|cc|cc}
 \toprule
	 & \multicolumn{2}{|c|}{$n=100$} & \multicolumn{2}{|c|}{$n=300$} & \multicolumn{2}{c}{$n=1000$} \\
        \cmidrule{2-7}
      \# sampled data  & $41^3$ & 178 & $41^3$ & 331 & $41^3$ & 643 \\
        \midrule
		$f_1$ & 8.27e-06 & 2.25e-05 & 1.51e-06 & 4.20e-06 & 3.62e-07 & 7.48e-07  \\
		$f_2$ & 4.42e-05 & 1.68e-04 &  8.14e-06 & 2.18e-05 & 1.87e-06 & 4.11e-06 \\
		$f_3$ & 1.24e-05 & 3.79e-05 & 3.77e-06 & 9.41e-06 & 1.22e-06 & 2.53e-06  \\  
		$f_4$ & 2.93e-04 & 5.60e-04  & 1.43e-04 & 2.55e-04 & 1.16e-04 & 2.63e-04  \\  
		$f_5$ & 1.31e-04 & 3.46e-04 & 9.09e-05 & 1.66e-04 & 6.61e-05 & 1.20e-04    \\    
		$f_6$ & 1.24e-04 & 3.22e-04 & 7.02e-05 & 1.34e-04 & 5.18e-05 & 1.09e-04   \\  
		$f_7$ & 1.65e-02 & 5.29e-02 & 1.15e-02 & 1.71e-02 & 1.10e-02 & 1.85e-02   \\
		$f_8$ & 2.47e-03 & 8.28e-03 & 9.60e-04 & 1.94e-03 & 7.20e-04 & 1.19e-03  \\ 
		$f_9$ & 1.43e-04 & 3.84e-04 & 1.14e-04 & 2.01e-04 & 9.84e-05 & 3.95e-04   \\  
		$f_{10}$ & 3.21e-04 & 9.74e-04 & 2.31e-04 & 4.00e-04 & 2.04e-04 & 3.91e-04 \\
  \bottomrule   
	\end{tabular}
\end{table}

\begin{remark}
    The major computational burden for the results in Tables~\ref{2Dex1} and~\ref{3Dex2} is the denoise of the KB-splines to get LKB-splines which requires a large number of data points and spline values as the noises are everywhere over $[0, 1]^d$. The denoising step is independent of the data approximation in the sense that the 
    computation is not dependent on the measurement of function values. 
    When dimension $d\gg 2$ gets large, one has to use an exponentially increasing number of 
    points and KB-spline values by, say a tensor product spline method for denoising, 
    and hence, the computational cost will suffer the curse of dimensionality. 
    However, the denoising step can be pre-computed once for all and can be done in parallel.  In particular, obtaining these large amount of data values is done by computation, not by collecting these function values.  
    That is, once we have the LKB-splines, 
    the rest of the computational cost is no more than the cost of solving a least squares problem.  We leave the numerical results for $d>3$ in \cite{S24}.
\end{remark}  

To check the approximation order $O(1/n)$ in Theorem~\ref{mjlai2019},  we plot the approximation errors of several aforementioned functions based on pivotal point locations against $n$ using log-log scale. The results are shown in Figure~\ref{ConvgPlots}.  It is worthwhile to note that the slopes in these plots are associated with the exponent $\alpha$ in Theorem~\ref{mjlai2019n}. In other words, if the slope of a convergence plot for a function is smaller than $-1$, then we can numerically conclude that such a function belongs to the Kolmogorov-Lipschitz  class. If the slope $\alpha$ satisfies $-1<\alpha<0$, then we can numerically conclude that such a function belongs to KH class as the outer function $g$ belongs to $C^{0,\alpha}$. In other words, 
our computational method provides a numerical approach to check if a multidimensional continuous function 
is Kolmogorov-Lipschitz  or not. 

Finally, in Figure~\ref{NNrPlots}, we plot the number of pivotal locations against $n$, 
we can see that the number of 
pivotal locations increasing linearly with $n$ and the increasing rates (slopes) are at most $d$. That is, we only 
need $O(nd)$ data locations and $O(nd)$ LKB functions to approximate a multidimensional continuous function $f$ 
with approximation rate $O(1/n)$ when $f$ is KL.  Therefore, the curse of dimensionality is 
overcome when $d=2$ and $d=3$.  See \cite{S24} for numerical evidence for $d=4,5,6$. 

\begin{figure}[htbp] 
    \centering
    \begin{tabular}{ccc}

    \scriptsize{$f(x,y)=x$} & \scriptsize{$f(x,y)=(1+2x+3y)/6$} &
    \scriptsize{$f(x,y)=1/(3+\sin(\sin(\sin(\sin(x^2-y^2)))))$} \\
    \includegraphics[width=0.30\textwidth]{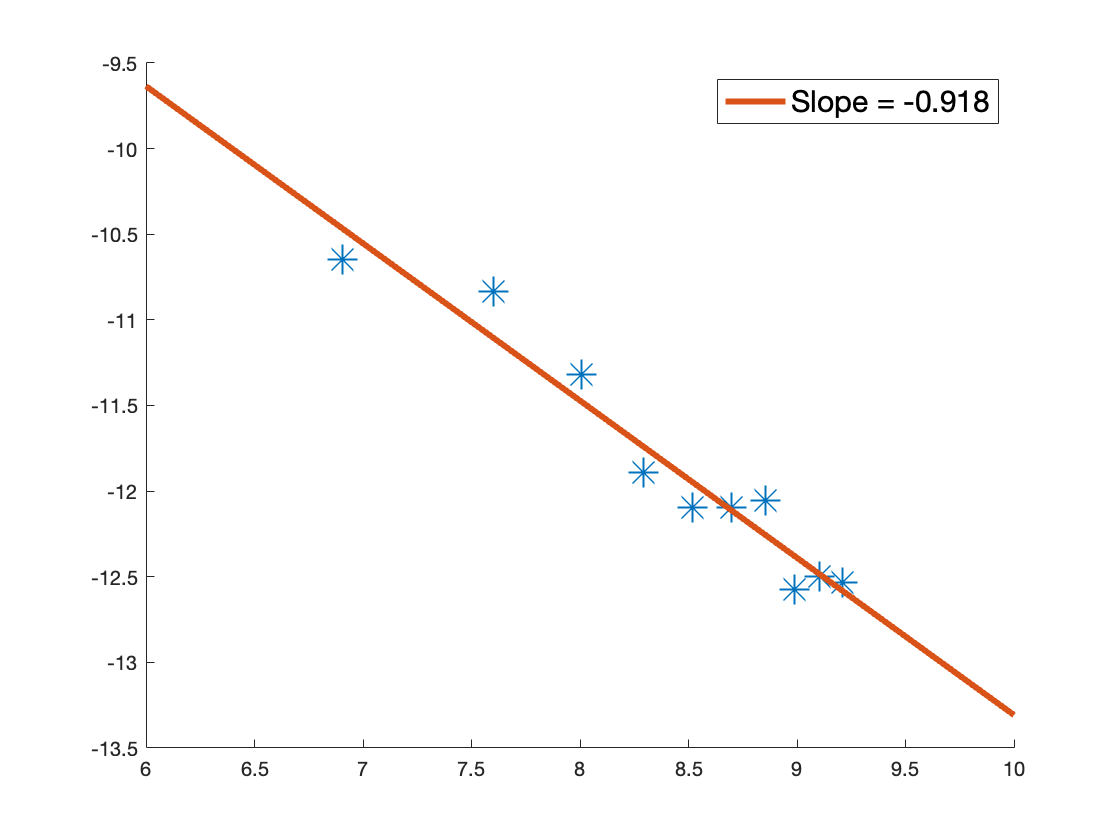} &
    \includegraphics[width=0.30\textwidth]{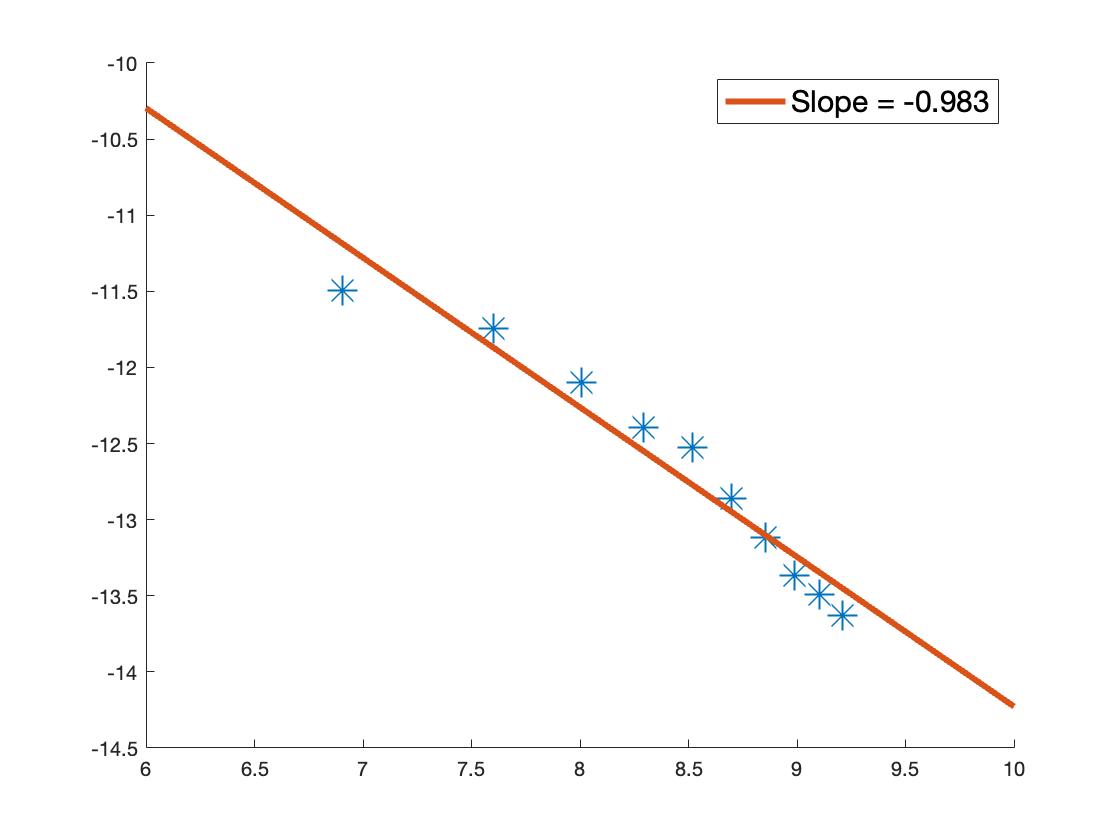} &
    \includegraphics[width=0.30\textwidth]{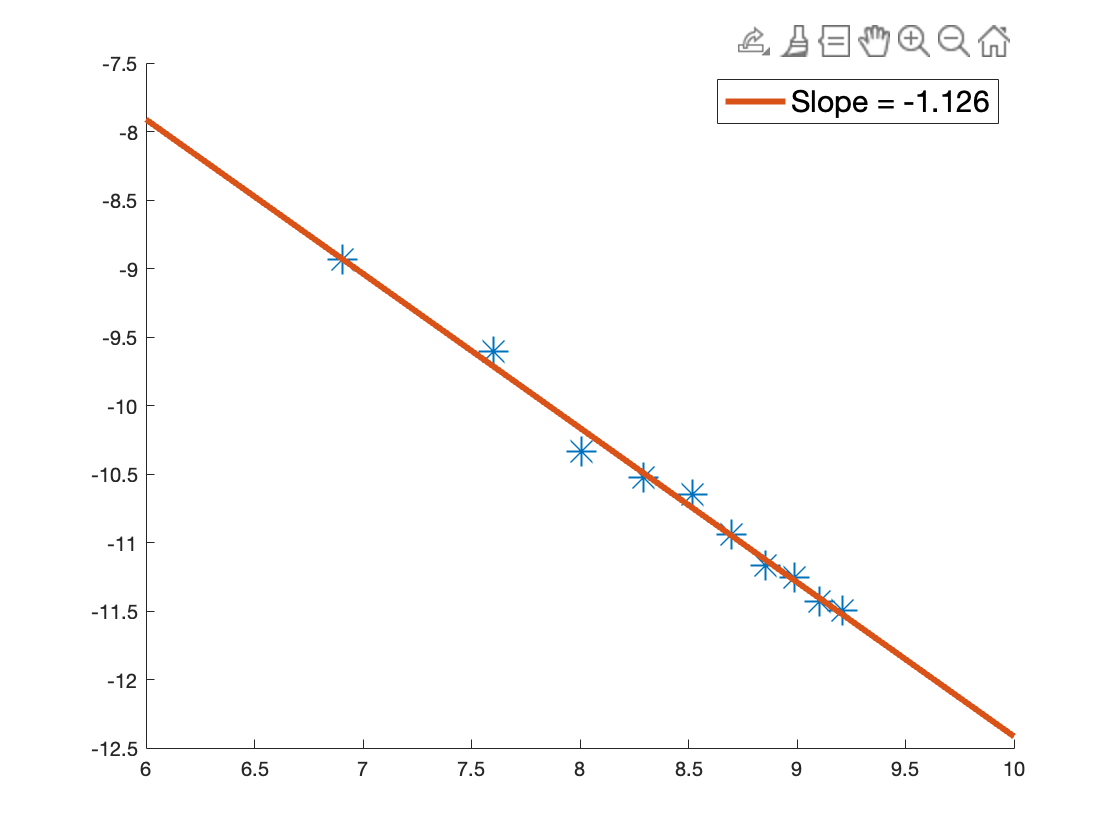}  \\

    \scriptsize{$f(x,y,z)=(1+2x+3y+4z)/10$} & \scriptsize{$f(x,y,z)=(x^2+y^2+z^2)/3$} &
    \scriptsize{$f(x,y)=(xy+yz+zx)/3$} \\
    \includegraphics[width=0.30\textwidth]{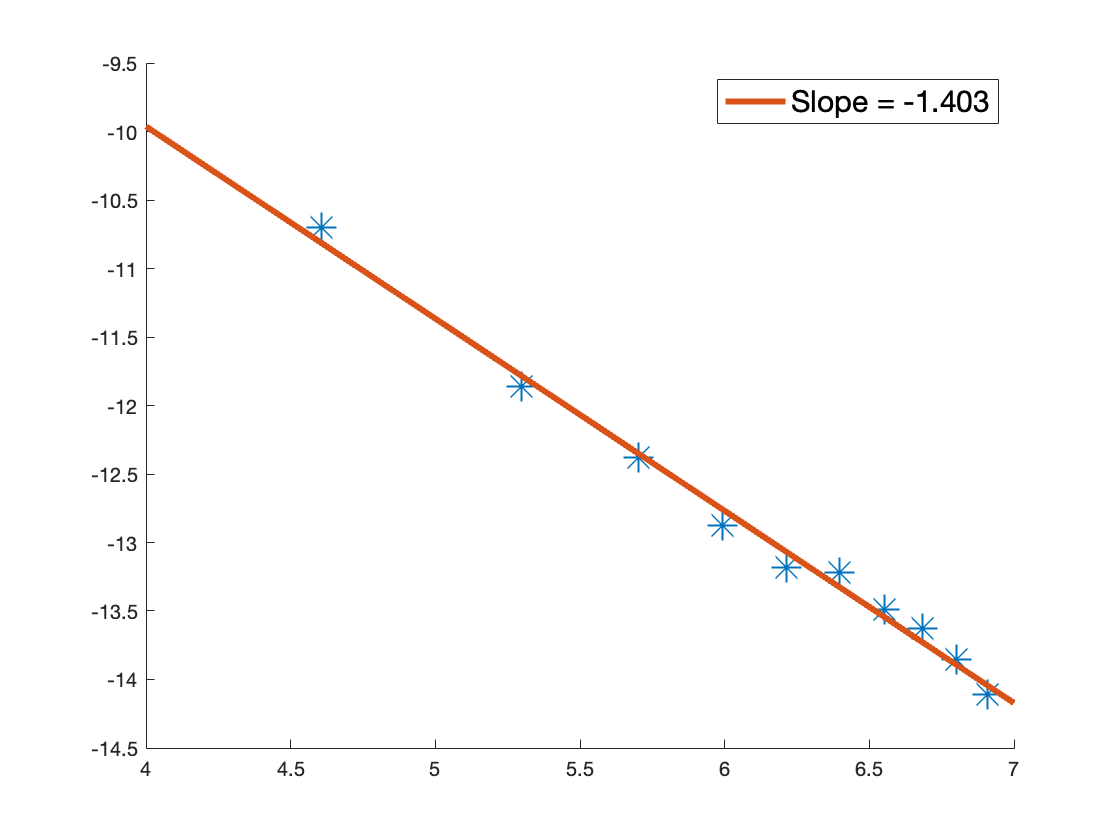} &
    \includegraphics[width=0.30\textwidth]{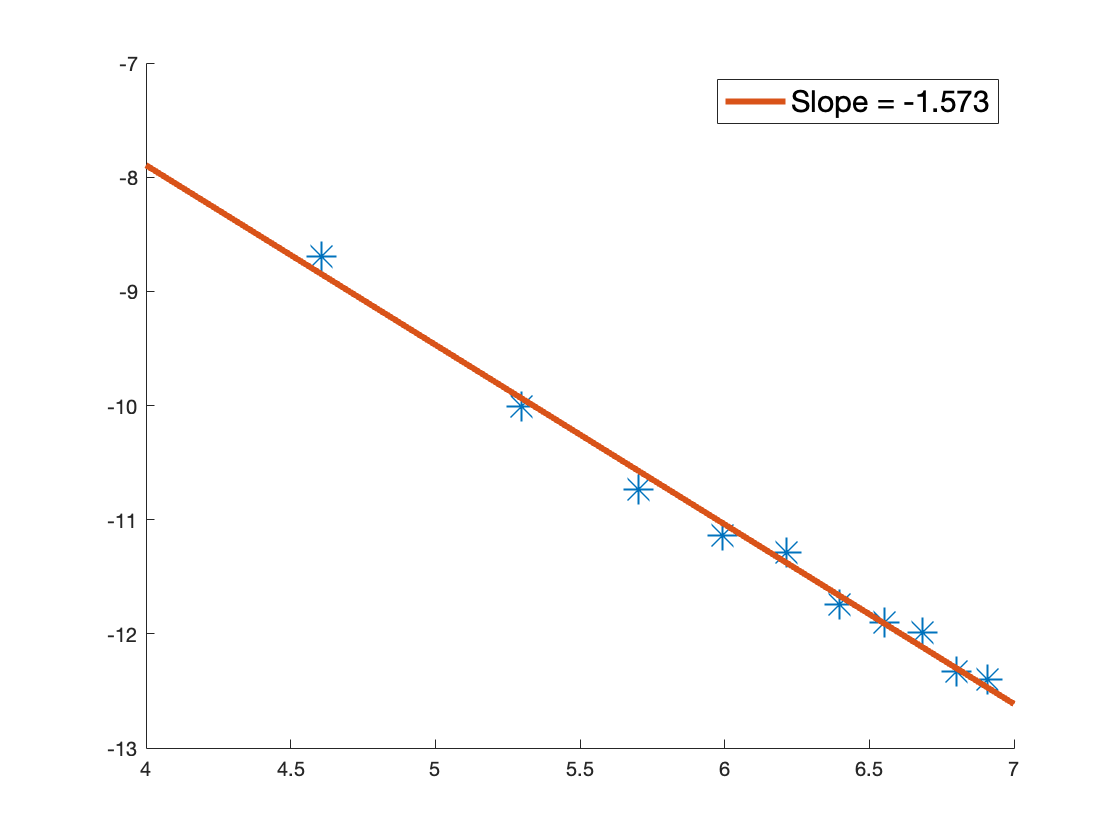} &
    \includegraphics[width=0.30\textwidth]{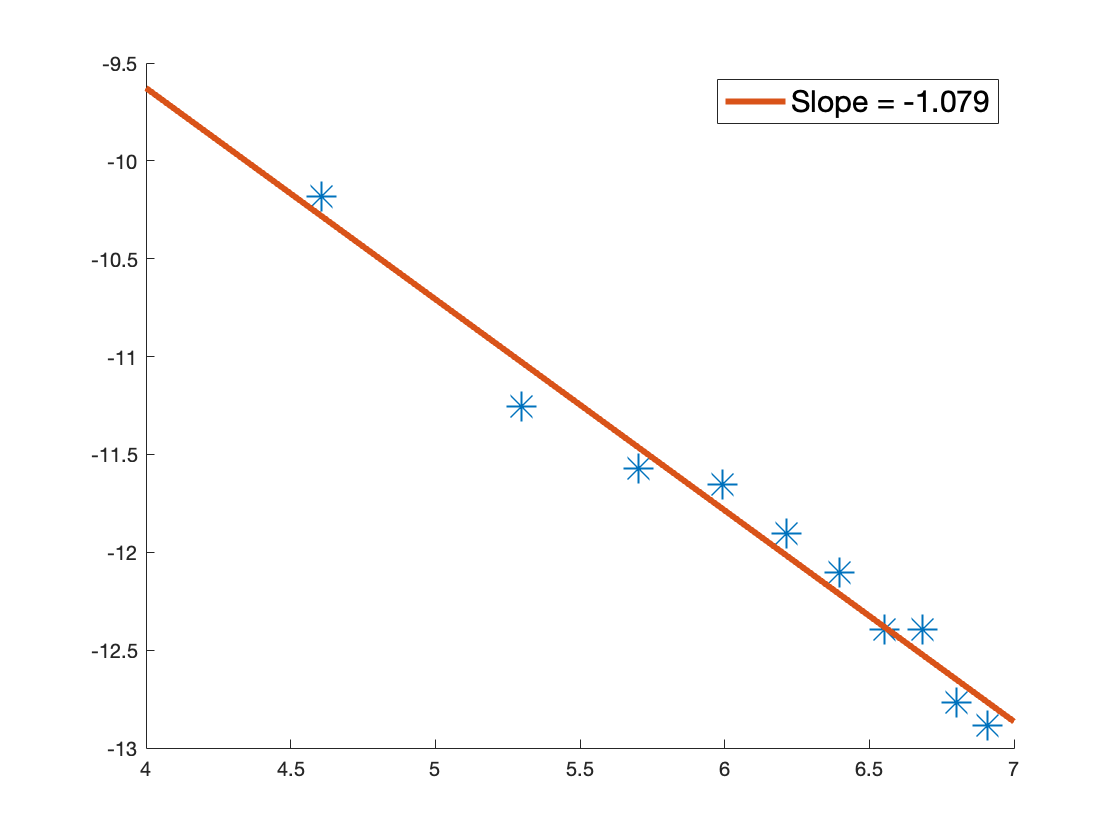} 
    \end{tabular}
    
    \caption{Plot of Convergence Rate using Log-log Scale for Functions in 2D and 3D. \label{ConvgPlots}}
\end{figure}

\begin{figure}[htbp] 
    \centering
    \begin{tabular}{cc}
    \includegraphics[width=0.4\textwidth]{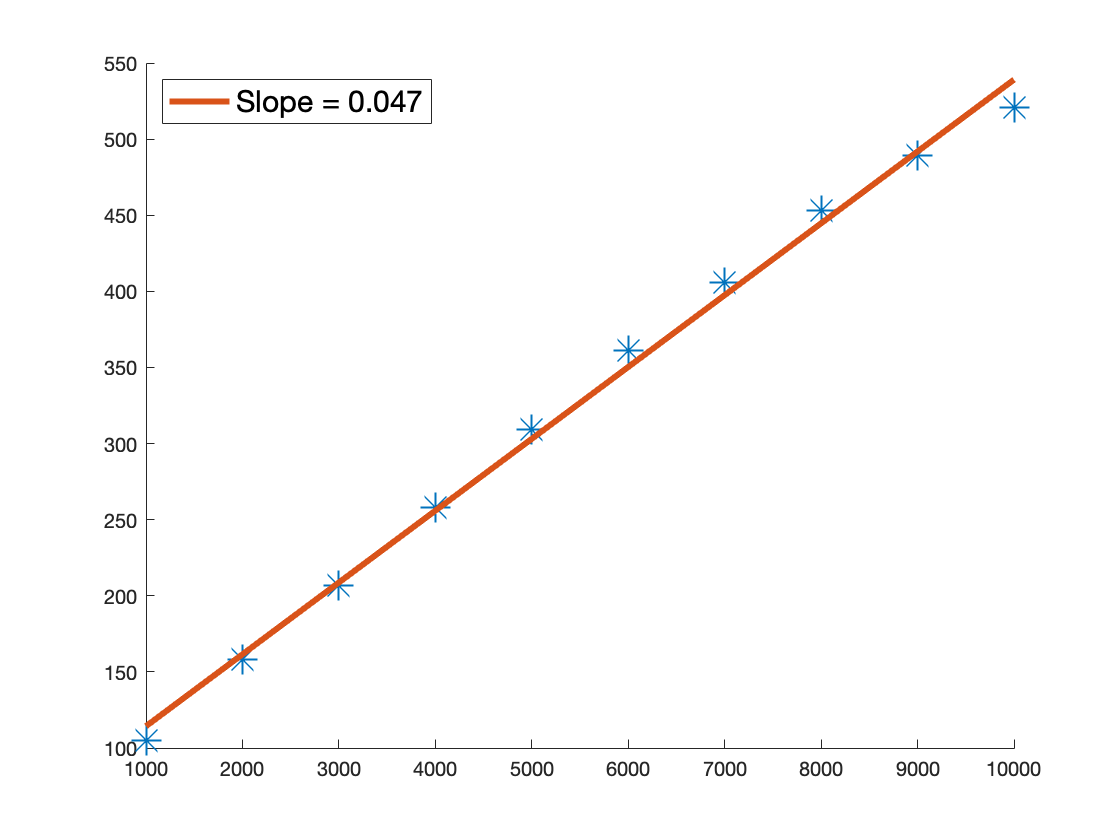} &
    \includegraphics[width=0.4\textwidth]{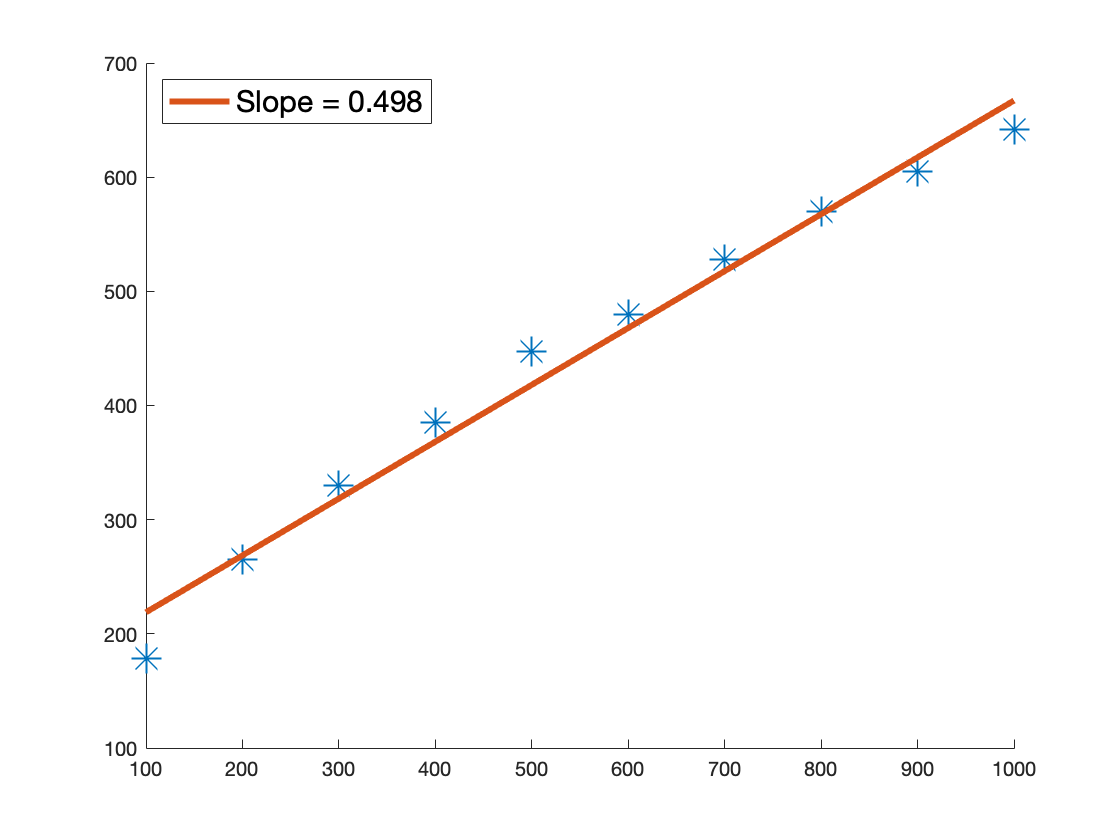}
    \end{tabular}
    
    \caption{Number of Pivotal Locations (vertical axis) against $n$ (horizontal axia) in 2D (left) and in 
    3D (right). \label{NNrPlots}}
\end{figure} 

Finally, let us end up this section with some important remarks. 

% \begin{remark}
% The major computational cost comes from denoising  
% \end{remark}
\begin{remark}
The pivotal data set is dependent on the degree of KB-splines and the LKB-splines, which are also dependent on the smoothing 
parameters and triangulation for converting KB-splines to LKB-splines. After LKB-splines are constructed, the 
pivotal point set is dependent on the discrete least squares (DLS) fitting method. 
For example, if we use randomly sampled points over $[0, 1]^d$ for a DLS method instead of equally-space points, the
pivotal point set is clearly different from the pivotal points based on the equally-spaced points over $[0, 1]^d$.
If we use $201\times 201$ equally-spaced points instead of $101\times 101$ equally-spaced points when constructing 
a DLS fitting based on LKB-splines over $[0, 1]^2$, the location of the pivotal data are different and the size 
of pivotal data set is slightly bigger than the ones shown on the left panel in Figure~\ref{pivotal}. 
\end{remark}

\begin{remark}
Certainly, there are many functions such as $f(x,y)= \sin(100 x)\sin(100 y)$ or 
$f(x,y)=\tanh(100((2x-1)^2+(2y-1)^2-0.25))$  which LKB-splines can not approximate well
based on the pivotal data sets above.  Such highly oscillated functions are hard to approximate no matter what methods are used.  We believe that
these functions will have a large Kolmogorov-Lipschitz constant. 
One indeed needs a lot of the data (points and the function values over the points) in order to approximate them well. One may also consider to use Fourier basis 
as outer functions rather than B-splines basis to approximate such highly oscillated trigonometric functions via KST. We leave it as a future research topic.        
\end{remark}

\begin{remark}
    To reproduce the experimental results in this paper, we uploaded our MATLAB codes in \url{https://github.com/zzzzms/KST4FunApproximation}. In fact, we have tested more than 100 functions in 2D and 3D with pivotal data sets which 
    enables us to approximate these functions very well.  
\end{remark}

\leftline{\bf Acknowledgement:} 
The authors would like to thank the anonymous reviewers for their comments and suggestions which improve the readability of the paper.

\end{document}